\documentclass{article}
\usepackage{fullpage}
\usepackage{setspace}
\usepackage{amsmath}
\usepackage{amsfonts}
\usepackage{graphicx}
\usepackage{hyperref}
\usepackage{courier}
\usepackage{subfigure}
\usepackage{xcolor}
\usepackage[english]{babel}
\usepackage{mathbbol}
\usepackage{multirow}
\usepackage{subfigure}
\usepackage{latexsym}
\usepackage{pdfsync}
\usepackage{epsfig}

\definecolor{Blue}{rgb}{0.3,0.3,0.9}
\usepackage{amsmath}
\usepackage{amsthm,amsbsy}
\usepackage{mathcomp}
\usepackage{textcomp}

\usepackage{amssymb}
\usepackage{graphicx}
\usepackage{graphicx}
\usepackage{dcolumn}
\usepackage{bm}
\usepackage{amsfonts}
\usepackage{latexsym}
\usepackage{pdfsync}
    \usepackage{fullpage} 
\usepackage{colordvi}
\usepackage{color}
\usepackage{booktabs}
\usepackage{graphicx}
\usepackage{stmaryrd}
\usepackage{cite}
\usepackage[linesnumbered,commentsnumbered,ruled]{algorithm2e}
\usepackage{epstopdf}
\input xy
\xyoption{all}

\newtheorem{thm}{Theorem}[section]

\newtheorem{prop}[thm]{Proposition}

\newtheorem{rem}[thm]{Remark}

\newcommand{\commentout}[1]{}

\newcommand{\nwc}{\newcommand}

\nwc{\bR}{\mb R}
\nwc{\bH}{{\mb H}}
\nwc{\bxp}{{{\mathbf x}}}
\nwc{\bap}{{{\mathbf y}}}

\nwc{\bPhi}{\mathbf{\Phi}}
\nwc{\bPsi}{\mathbf{\Psi}}
\nwc{\bh}{\mathbf h}

\nwc{\bI}{\mathbf I}
\nwc{\bP}{\mathbf P}
\nwc{\bs}{\mathbf s}
\nwc{\bd}{\mathbf{d}}
\nwc{\bX}{\mathbf X}

\nwc{\om}{\omega}

\nwc{\nwt}{\newtheorem}
\nwc{\xp}{{x^{\perp}}}
\nwc{\yp}{{y^{\perp}}}
\nwt{remark}{Remark}
\nwt{definition}{Definition}
\nwt{example}{Example}
\nwt{corollary}{Corollary} 

\nwc{\ba}{{\mb a}}
\nwc{\bal}{\begin{align}}
\nwc{\ben}{\begin{equation*}}
\nwc{\beqq}{\begin{equation}}
\nwc{\bea}{\begin{eqnarray}}
\nwc{\beq}{\begin{eqnarray}}
\nwc{\bean}{\begin{eqnarray*}}
\nwc{\beqn}{\begin{eqnarray*}}
\nwc{\beqast}{\begin{eqnarray*}}

\nwc{\eal}{\end{align}}
\nwc{\een}{\end{equation*}}
\nwc{\eeqq}{\end{equation}}
\nwc{\eea}{\end{eqnarray}}
\nwc{\eeq}{\end{eqnarray}}
\nwc{\eean}{\end{eqnarray*}}
\nwc{\eeqn}{\end{eqnarray*}}
\nwc{\eeqast}{\end{eqnarray*}}

\nwc{\vep}{\varepsilon}
\nwc{\ep}{\epsilon}
\nwc{\ept}{\epsilon}
\nwc{\vrho}{\varrho}
\nwc{\orho}{\bar\varrho}
\nwc{\ou}{\bar u}
\nwc{\vpsi}{\varpsi}
\nwc{\lamb}{\lambda}
\nwc{\Var}{{\rm Var}}

\nwt{proposition}{Proposition}
\nwt{theorem}{Theorem}
\nwt{summary}{Summary}
\nwt{lemma}{Lemma}
\nwc{\nn}{\nonumber}
\nwc{\mf}{\mathbf}
\nwc{\mb}{\mathbf}
\nwc{\ml}{\mathcal}

\nwc{\IA}{\mathbb{A}} 
\nwc{\bi}{\mathbf i}
\nwc{\bo}{\mathbf o}
\nwc{\IB}{\mathbb{B}}
\nwc{\IC}{\mathbb{C}} 
\nwc{\ID}{\mathbb{D}} 
\nwc{\IM}{\mathbb{M}} 
\nwc{\IP}{\mathbb{P}} 
\nwc{\II}{\mathbb{I}} 
\nwc{\IE}{\mathbb{E}} 
\nwc{\IF}{\mathbb{F}} 
\nwc{\IG}{\mathbb{G}} 
\nwc{\IN}{\mathbb{N}} 
\nwc{\IQ}{\mathbb{Q}} 
\nwc{\IR}{\mathbb{R}} 
\nwc{\IT}{\mathbb{T}} 
\nwc{\IZ}{\mathbb{Z}} 

\nwc{\cE}{{\ml E}}
\nwc{\cP}{{\ml P}}
\nwc{\cQ}{{\ml Q}}
\nwc{\cL}{{\ml L}}
\nwc{\cX}{{\ml X}}
\nwc{\cW}{{\ml W}}
\nwc{\cZ}{{\ml Z}}
\nwc{\cR}{{\ml R}}
\nwc{\cV}{{\ml V}}
\nwc{\cT}{{\ml T}}
\nwc{\crV}{{\ml L}_{(\delta,\rho)}}
\nwc{\cC}{{\ml C}}
\nwc{\cO}{{\ml O}}
\nwc{\cA}{{\ml A}}
\nwc{\cK}{{\ml K}}
\nwc{\cB}{{\ml B}}
\nwc{\cD}{{\ml D}}
\nwc{\cF}{{\ml F}}
\nwc{\cS}{{\ml S}}
\nwc{\cM}{{\ml M}}
\nwc{\cG}{{\ml G}}
\nwc{\cH}{{\ml H}}
\nwc{\bk}{{\mb k}}
\nwc{\bn}{{\mb n}}
\nwc{\cbz}{\overline{\cB}_z}
\nwc{\supp}{{\hbox{supp}}}
\nwc{\fR}{\Re}
\nwc{\bY}{\mathbf Y}

\nwc{\pft}{\cF^{-1}_2}
\nwc{\bU}{{\mb U}}
\nwc{\bG}{{\mb G}}
\nwc{\bg}{\mathbf{g}}
\nwc{\mbf}{\mathbf{f}}
\nwc{\mbe}{\mathbf{e}}
\nwc{\be}{\mathbf{e}}
\nwc{\Om}{\Omega}
\nwc{\ind}{\operatorname{I}}
\nwc{\mbx}{\mathbf{f}}
\nwc{\bb}{\mathbf{g}}
\nwc{\xmax}{f_{\rm max}}
\nwc{\xmin}{f_{\rm min}}
\nwc{\suppx}{\hbox{\rm supp} (\mbf)}
\nwc{\by}{\mathbf{h}}
\nwc{\bZ}{\mathbf{Z}}
\nwc{\bF}{\mathbf{F}}
\nwc{\bE}{\mathbf{E}}
\nwc{\bV}{\mathbf{V}}
\nwc{\cI}{\IZ^2_N}
\nwc{\chis}{{\chi^{\rm s}}}
\nwc{\chii}{{\chi^{\rm i}}}
\nwc{\pdfi}{{f^{\rm i}}}
\nwc{\pdfs}{{f^{\rm s}}}
\nwc{\pdfii}{{f_1^{\rm i}}}
\nwc{\pdfsi}{{f_1^{\rm s}}}
\nwc{\thetatil}{{\tilde\theta}}
\nwc{\red}{\color{red}}
\nwc{\prox}{\hbox{prox}}
\nwc{\diag}{\hbox{\rm diag}}
\nwc{\sloc}{J_{\rm f}}
\nwc{\bu}{\xi}
\nwc{\bv}{\eta}
\nwc{\cU}{\mathcal{U}}
\nwc{\cN}{\mathcal{N}}
\nwc{\bN}{\mathbf{N}}
\nwc{\mbm}{\mathbf{m}}
\nwc{\bw}{\mathbf{w}}
\nwc{\im}{i}
\nwc{\bom}{\mathbf{w}}
\nwc{\bt}{\mathbf{t}}
\nwc{\z}{y}
\nwc{\cY}{\mathcal{Y}}

\title{
 Arnoldi algorithms  with  structured orthogonalization }
\author{The Author}
\date{}                                           
\author{Pengwen Chen$^*$, Chung-Kuan~Cheng$^{\dag\ddag}$,
Xinyuan Wang$^{\ddag}$\\
\normalsize $^*$ Applied mathematics, National Chung Hsing University, Taiwan\\
\normalsize $^\dag$CSE and $^\ddag$ECE Departments, UC San Diego, La Jolla,CA, USA\\
\normalsize email: pengwen@nchu.edu.tw, ckcheng@ucsd.edu, xiw193@eng.ucsd.edu
}

\begin{document}
\maketitle

\begin{abstract}
We study  a stability preserved Arnoldi algorithm for matrix exponential in  the time domain simulation of large-scale power delivery networks (PDN), which are formulated as semi-explicit differential algebraic equations (DAEs). 
The solution can be decomposed to a sum of two projections, one in the range of the system operator
 and the other in its null space.  
The range projection can be  computed with
one shift-and -invert Krylov subspace method. The other projection can be computed with the algebraic equations. Differing from the ordinary Arnoldi method, the orthogonality in the Krylov subspace is replaced with the semi-inner product induced by the positive semi-definite system operator. With proper adjustment,   numerical ranges of the Krylov operator lie in the right half plane, and  we obtain  theoretical convergence analysis  for the modified  Arnoldi algorithm in computing phi-functions. 
Lastly,   simulations on RLC networks are demonstrated to validate the effectiveness of the  Arnoldi algorithm  with structured-orthogonalization.
\end{abstract}

\section{Introduction}
Evaluating the  performance of a power deliver network (PDN) has become a critical issue in   very large-scale integration (VLSI) designs.
  The power supply from the package down to on-chip integrated circuits is distributed through metal layers and vias, which could be modeled as a linear network consisting of resistors, capacitors and inductors \cite{nassif2008power}. The on-chip circuit modules are simplified as time-varying current sources in PDN analysis. Due to the shrinking feature size and increasing design complexity, the network could easily consist of millions to billions of elements which result in an extremely huge system. Moreover, the values of elements in a system level PDN may vary greatly and the transient responses include many different scaled time constants, which makes the whole differential system very stiff.  In order to characterize the long term dynamic behavior, an extended time span at small-scaled time steps is necessary and extra computation efforts are required. At the same time, the stiffness of the system is increased which degrades the performance of traditional simulation methods. All the challenges make a fast and accurate simulator in high demand.

Let $x(t)\in \IR^N$ be  the solution to a   system of stiff differential equations,\cite{chen2018transient, AWESOME} 
  \beqq\label{sys0}
  \frac{dq(t)}{dt}+f(x(t))=u(t), \; x(0)=x_0,
  \eeqq
   where  $u(t)$ is the input signal to the circuit system, 
$x(t)\in \IR^N$ of large dimension $N$ denotes nodal voltages and branch currents at time $t$ and $q,f\in \IR^N$ are the charge(or flux) and current (or voltage) terms, respectively. 
  The system is governed by Kirchhoff's current law and voltage law. 
  With   linearization, 
  we have
\beqq\label{sys}
C \frac{dx}{dt}+G x=u(t), x(0)=x_0,\eeqq
where  $C$ and $G$ both are $N\times N$ matrices, which are the Jacobian matrices of $q$ and $f$ with respect to $x$, respectively. 
In the study, we  assume that 
 $C,G$ are  constant  matrices  and
\begin{equation}\label{assGC}
    \begin{cases}
    &\textrm{ $G$ is positive definite, but not necessarily symmetric}; \\
& \textrm{ $C$ is positive semi-definite  and symmetric. }
    \end{cases} 
 \end{equation}
 Every node is supposed to connect to power or ground
via a path of resistors, which makes $G$ nonsingular.
For a stiff  system,  the solution can  be  of multiple timescales, i.e., the attractive solution is surrounded with fast-changing nearby solutions.

%
%
When $C$ is nonsingular, the solution can be formulated as  exponentials of the matrix $A:=C^{-1} G$.  There are various ways to implement the computation\cite{Moler78},\cite{Moler03}  depending on the state companion matrix $A$. 
When $A$ is a matrix in small size, 
the most effective algorithm is a scaling-and-squaring method based on Pad{\'e} approximation\cite{Tref12}.
When 
$A$ is sparse and large,
 one general and well-established technique is 
approximating the action of the matrix exponentials  
 in the
class of Krylov subspaces. 
 One essential ingredient  is  the
evaluation or approximation of the product of the exponential of the Jacobian  $A$ with a
vector $v$.
The application  of Krylov subspace techniques  has been actively
investigated in the literatures\cite{Fri_1989, Saad92,Moler03,Hoch09,Wright_phi_2012,JIMENEZ2020112758}.
In general, the  nonlinear form in (\ref{sys0}) can be numerically handled by various exponential Runge-Kutta schemes with the aid of exponential integrators\cite{Hoch09,hochbruck_ostermann_2010} and 
references therein.

It is well-known that  Krylov
subspace methods for matrix functions exhibits super-linear
convergence behavior under  sufficient large Krylov dimension (larger than the norm of the operator)\cite{Saad92}\cite{Hoch}. Recently,  researchers  observe  the
superiority  of rational Krylov subspace methods over standard Krylov subspace methods, in particular, the spectrum of the operator lies in the half-plane, e.g., the Laplacian operators in PDEs\cite{Drus98},\cite{Grim08}. 
The convergence of
computing exponential integrators of  evolution equations in the resolvent Krylov subspace 
 is independent of the operator norm of $A$ from one numerical discretization, when $A$ in $\exp(-A)$ has numerical range(or called field of values) in the right half plane~\cite{Grim12}\cite{Tanj17}.

Exponential  integrator based methods have been introduced 
 for PDN transient simulations~\cite{zhuang2016simulation,chen2018transient}. Compared to the traditional  linear multi-step methods, the matrix exponential based method is not bounded by the Dahlquist stability barrier thus larger step size can be employed~\cite{wanner2006dahlquist, zhuang2016simulation}. 
The stability of matrix exponential based method when applied to ODEs has been well established in previous work \cite{Weng12_TCAD, zhuang2016simulation}. For general circuit simulation with DAEs, the stability remains an interesting topic \cite{freund2000krylov, ilchmann2014surveys, winkler2003stochastic, takamatsu2010index}. Numerical stability issues are reported in \cite{chen2018transient, AWESOME} and reveal the limitation of matrix exponential computations with Krylov subspace. Similar problems occur in the eigenvalue computation \cite{IRA, nour1987implement} and model order reduction for interconnect simulation \cite{rommes2009exploiting, IRA2}, where Krylov subspace methods are widely used. 
As one shift-and-invert method,  one modified Arnoldi algorithm 
for matrix exponential are proposed 
to provide  stable computations of matrix exponentials, where
 Arnoldi vectors are 
orthogonal 
with respect to the system operator $C$\cite{chen2018transient, AWESOME}.

In this paper, we shall examine
 the modified shift-and-invert Arnoldi algorithm from the perspective of numerical ranges, which provides one theoretical foundation for the Arnoldi algorithm described in~\cite{2018transient, AWESOME}.    
 Since the matrix $C$ could be singular in PDN transient simulation, we introduce $C$ semi-inner product as well as its induced norm,   \[ \langle x, y\rangle_C:=\Re(x^* C y), \; \|x\|_C:=\Re(x^* C x) \] to derive the error analysis, instead  of the ordinary  inner product $\langle  x,y\rangle:=\Re(x^* y)$.  Likewise, the $C$-norm $\|x\|_C:=\sqrt{x^* Cx}$ is used to define the so-called $C$-numerical ranges in (\ref{FC}). The advantage of $C$ semi-inner product introduced in the modified Arnoldi algorithm is two-fold:  the null-space component is removed in the Arnoldi iterations and the $C$-numerical range of the operator in the matrix exponentials lies in the right half plane. The numerical range of the upper Hessenberg matrix is properly restricted within a disk with center at $1/2$ and radius $1/2$.  The $C$ semi-inner product as well as the associated Arnoldi algorithm 
  have been employed  for different purposes, e.g,  solving generalized eigenvector problems\cite{ERICSSON, IRA} and  generating stable and passive Arnoldi-based model order reduction\cite{IRA2}.
  
The main contributions are  listed as follows.
With the aid of eigenvectors of $C$ as a basis,
solutions $x(t)$ to PDNs can be  decomposed to a sum of $x_\cR(t)$ and $x_\cN(t)$.
The shift-and-invert  Krylov method in \cite{AWESOME}  computes $x_\cR(t)$, which actually 
captures the dominant transient dynamical behaviors.  The orthonormal basis of Krylov subspace is generated by a quadratic norm with the system matrix to preserve the passivity property of the system, which yields  stable transient simulations. 
The positive definite matrix $G$ guarantees   the $C$-numerical range of $G^{-1} C$ lying the right half plane, which establishes  the convergence to $x_\cR(t)$ as Krylov dimension tends to infinity, including  posterior error bounds and  prior error bounds.
 The shift parameter $\gamma$ in the shift-and-invert method  provides the  flexibility to confine the spectrum of ill-conditioned systems \cite{ericsson1980spectral}.   The error with  
 $\varphi_k$-functions tends to $0$ as the dimension increases.  In the case of $\varphi_0$ computation with $\gamma$ proportional to time step size, the error curve with respect to $\log \gamma$ is a $\cap$-shaped curve. The stagnation in the small $\gamma$ can be significantly improved, when the $\varphi_1$ or $\varphi_2$ computation is introduced,  which is consistent with empirical studies  reported in (\cite{AWESOME}).

The rest of this paper is organized as follows. The differential algebraic equations(DAEs) framework is introduced in Sec.~\ref{Sec_back}. The explicit formulations of solutions in the basis of eigenvectors of $G^{-1} C$ and in the basis of eigenvectors of $C$ are given in section~\ref{Exact_sol_1} and~\ref{Exact_sol_2}, respectively. In the paper, we focus on the computation of the projected solution $x_\cR(t)$. In section~\ref{Krylov}, we introduce Krylov space corresponding to the shift-and invert method, which is used to approximate the solution. In section~\ref{Posterior}, we give a posterior  error bound based on the residual errors and prior error bound. 
In section~\ref{simulations}, we provide simulations on  RLC networks with $G$ only positive semidefinite to validate the effectiveness of the modified shift-and-invert Arnoldi algorithm and examine the error behaviors in computing matrix exponentials. 
%

\subsection{Solutions of nonsingular systems }\label{Sec_back}
 
Suppose that  $C$ is  nonsingular with $A=C^{-1}G$.  The variation-of-constants formula yields the solution $x(t)$ described by 
\beqq\label{1sol}
x(t)=\exp(-tA) x_0+\int_0^t \exp(-(t-s)A) C^{-1} u(s)\,  ds.
\eeqq
Introducing so-called phi-functions,\beqq\label{phix_5}
\varphi_0(z):=\exp(z),\;\varphi_{k+1}(z):=(\varphi_k(z)-(k!)^{-1})/z \textrm{  for $k\ge 0$ },
\eeqq
we can approximate  (\ref{1sol}) under linearization on the source term $C^{-1} u(s)\approx b+b's$ 
 as  a sum of the $\varphi_0$, $\varphi_1$ and $\varphi_2$ terms. 
\beqq\label{phi_x}
x(t+h)\approx \varphi_0(-hA ) x(t)+h \varphi_1(-hA) b+h^2 \varphi_2(-hA)  b',
\eeqq
where $\varphi_0(z)=\exp(z)$ and $\varphi_1(z)=z^{-1}(\exp(z)-1)$.
 One can employ  the shift-and-invert Arnoldi transform to solve one nonsingular differential system as in (\cite{Botchev})
Briefly, let $A=C^{-1}G$ and construct 
the Krylov subspace with respect to $(I+\gamma A)^{-1}$ with a parameter $\gamma>0$, i.e., \[
(I+\gamma A^{-1})^{-1} V_m =V_{m+1} \widetilde H_m,
\]
where  one orthogonal basis matrix $V_m\in \IR^{N\times m}$ and one
 upper-Hessenburg matrix $\widetilde H_m\in \IR^{(m+1)\times m}$ are generated.
 Let $H_m$ be the sub-matrix of $\widetilde H_m$ without the last row. Then the  terms $\varphi_0, \varphi_1$ in (\ref{1sol}) can be approximated by the exponential function of $H_m$, e.g., \beqq\label{solap}
 \exp(-tA) x_0\approx  \|x_0\| \exp(-\gamma t (H_m^{-1}-I_m)) e_1.
 \eeqq

\subsection{Solutions of singular systems } \label{Exact_sol_1}
 A nonsingular matrix $C$ cannot always achieved in general power delivery networks.
  For instance, the nodes without nodal capacitance or inductance would contribute to the algebraic equations and the corresponding matrix $C$ is not invertible. 
 One major impact from the singularity  is that the system in (\ref{sys}) is in fact one combination of differential equations and algebraic equations, i.e.,   $x(t)$ must satisfy the  range condition: $x(t)-G^{-1}u(t)$ in the range of $G^{-1} C$.  
  In addition, since the projection $H_m$ is constructed from an initial vector, without careful and proper handling,   the matrix $H_m$ could become a nearly  degenerate matrix, and (\ref{solap}) boils down to be an erroneous approximation. Hence, it is natural to perform some proper decomposition on $x(t+h)$ based on  nonzero and zero eigenvalues, so that  $H_m$ is not contaminated by null vectors and the solution $x(t)$ can be computed accurately. 
  
We discuss two decompositions to express the solutions. 
Start with the standard approach in differential equations.  (This approach is  listed as  Method 16  in~\cite{Moler03}.)
Let $G^{-1}C=V \Lambda V^{-1}$ be the Joran canonical form decomposition of $G^{-1}C$, where  
\[
\Lambda=\left(
\begin{array}{cc}
 J_\cR & 0 \\
0  & J_\cZ 
\end{array}
\right)\in \IC^{N\times N}
\]
is in Jordan normal form.
 The submatrix  $J_\cR\in \IC^{r\times r}$ consists of a few Jordan blocks  corresponding to  nonzero eigenvalues of $G^{-1}C$ and $J_\cZ\in \IR^{(N-r)\times (N-r)}$ is a nilpotent matrix corresponding to  eigenvalue zero of $G^{-1} C$.  Since the null space of $G^{-1}C$ has dimension $N-n$, the algebraic multiplicity of  the  eigenvalue zero  is not less than $N-n$.
 Write $V=[V_\cR, V_\cZ],\; V_\cZ:=[V_\cG, V_\cN ]$, where  columns of $V_\cR$ and $V_\cZ$ are the (generalized) eigenvectors  of nonzero eigenvalues, respectively.  Columns of $V_\cG $ and $ V_\cN$ are the  generalized eigenvectors  and the eigenvectors of  eigenvalue $0$. That is, columns of $V_\cN$ are the null vectors of $G^{-1} C$. 
 Let $U:=(V^{-1})^*=[U_\cR,  U_\cZ]$, where $A^*$ is the Hermitian transpose of a matrix $A$.  Consider  the solution decomposition,  \beqq \label{apx} x(t)=x_\cR(t)+x_\cZ(t)= V_\cR x_1 (t)+V_\cZ x_2 (t)\eeqq  with some vector functions $x_1 (t), x_2(t)$.
Let 
\beqq\label{GC1}
U^* G^{-1} C V=
\left(
\begin{array}{cc}
 J_\cR & 0 \\
0  & J_\cZ  
\end{array}
\right).
\eeqq
Multiplying with $U^* G^{-1}$
on   (\ref{sys}) yields one differential equation for $x_1$
 \beqq\label{eq15}
J_\cR   \frac{dx_1}{dt}+  x_1=U_\cR^* G^{-1}u(t)
  \eeqq
  and
 \beqq\label{eq15''}
J_\cZ   \frac{dx_2}{dt}+  x_2=U_\cZ^* G^{-1}u(t).
  \eeqq

Focus on (\ref{eq15''}) first. For simplicity, assume that $G^{-1}u(t)$ is a linear function in $t$, i.e. for some constant vectors $w_0, w_1$, \[ U_\cZ^* G^{-1}u(t)=w_0+w_1 t.\] 
The solution  $x_2(t)$ is also linear and can be expressed as 
\[
x_\cZ(t)= V_\cZ x_2(t)=V_\cZ(w_1 t+w_0-J_\cZ w_1)=V_\cZ( U_\cZ^* G^{-1}u(t)-J_\cZ U_\cZ^* G^{-1}\frac{d u(t)}{dt}).
\] 
Return  to (\ref{eq15}).  Let
$\widetilde u(t)=J_\cR^{-1}   U_\cR ^*  G^{-1}   u(t)$.
The solution $x_1(t)$ in  (\ref{eq15}) can be expressed as 
\beqq\label{sol2new}
x_{\cR}(t):=V_\cR x_1(t)= V_\cR  \left\{\exp(-t J_\cR ^{-1}) U_\cR ^* x(0)+\exp(-t J_\cR^{-1}) \int_0^t \exp(sJ_\cR ^{-1}) \widetilde u(s) \, ds\right\}.
\eeqq

\subsection{Solution decomposition under eigenvectors of $C$}\label{Exact_sol_2}
The matrices $V_\cR, U_\cR, J_\cR$ in (\ref{sol2new}) are generally complex-valued, which makes the computation for large PDN systems very challenging.  
Next, we introduce one set of  \textit{real} basis vectors to express the solution in (\ref{sys}), the eigenvectors of $C$. 
Let $C=V_C C_1V_C^\top $ be the eigenvector decomposition of $C$ with $C_1$ diagonal and singular. Let $P_C=V_C V_C^\top$ be the  orthogonal projection matrix on the range of $C$. Also  introduce orthogonal subspaces
$\cR$ and $\cN$,
\begin{eqnarray}
&&\cR:=\{ P_C  x : x\in \IR^N\}, \; \\
&&\cN:=\{ x\in \IR^N:  P_C x=0 \}.
\end{eqnarray}  We employ
 \beqq\label{V_def} V:=[V_\cR,  V_\cN], \; V_\cR=V_C,
 U:=[U_\cR,  U_\cN]=(V^{-1})^\top, 
 \eeqq to decouple the system in (\ref{sys}), where columns of 
$V_\cR \in \IR^{N\times n}$,$V_\cN \in \IR^{N\times (N- n)}$ are  basis vectors  in  $\cR$ and  $\cN$, respectively. 

 Write $G, C$ in block forms,
\beqq\label{GC}
U^\top GV=
\left(
\begin{array}{cc}
 G_1 & G_2 \\
G_3  & G_4  
\end{array}
\right),\;U^\top CV=
\left(
\begin{array}{cc}
 C_1 & 0 \\
 0  & 0  
\end{array}
\right),
\eeqq
where $C_1\in \IR^{n\times n}$ is non-singular, a positive definite and symmetric sub-matrix. 
  Consider  the following solution decomposition,  \beqq \label{apx1} x(t)=x_\cR(t)+x_\cN(t)= V_C x_1 (t)+V_\cN x_2 (t)\eeqq  with some vector functions $x_1 (t), x_2(t)$.
  Applying  $G^{-1}$ on (\ref{sys}) yields \textit{ one range  consistency constraint on $x(t)$ }  that   $x(t)-G^{-1}u(t)$ must lie in the range of $G^{-1}C$, including the initial vector $x(0)$. 
 Actually, from  (\ref{GC}), the system in  (\ref{sys}) is a combination of one differential  system and one algebraic system, i.e., 
\begin{eqnarray}
&&C_1 \frac{dx_1}{dt}=-G_1 x_1-G_2 x_2+(u)_1\label{eq1}\\
 &&G_3 x_1+G_4 x_2=(u)_2. \label{eq2}
\end{eqnarray}

Suppose $G_4$ is invertible.   With (\ref{eq2}), we can  eliminate $x_2$ in  (\ref{eq1}) and reach  one \textit{nonsingular} differential  system  of $x_1$, i.e., 
\begin{eqnarray}
C_1 \frac{dx_1}{dt}&=&-(G_1-G_2 G_4^{-1} G_3) x_1+G_2G_4^{-1}u_2+u_1\\
&=&-(G^{-1})_{1,1}^{-1} x_1+G_2G_4^{-1}u_2+u_1\label{eq61}.
\end{eqnarray}
 Such a system of differential-algebraic equations can also 
 occur in the simulation of mechanical multi-body systems, e.g.\cite{Drazin}. Finally, we can determine $x_\cN$, i.e., $x_2(t)$ from  (\ref{eq2}), if $G_4$ is invertible. 
 Hereafter we shall focus on the computation of $x_1(t)$. 
Keep in mind that the block form in (\ref{GC}) is only of theoretical interest, since the explicit formulation  requires the information of eigenvectors of $C$.  In practical applications of large dimension, the explicit formulation in  (\ref{eq1},\ref{eq2}) is unlikely to be known in advance. 

Next, we introduce one sufficient condition: assume the \textit{ positive definite } property on $G$, which ensures  the   invertibility of $ G_4:=V_\cN^\top G V_\cN$. 

\begin{prop}\label{prop1.2}
Assume that $C, G$ satisfy (\ref{assGC}) with   $v^\top Gv\ge \epsilon \|v\|^2$  for some positive scalar $\epsilon$. 
Let \beqq 
\label{Bdef}
B=G^{-1}C,\; B_{1,1}=V_C^\top  B V_C.\eeqq Then the matrix $B_{1,1}$ is invertible. In addition, the eigenvalue $\lambda$ of $B_{1,1}$ has positive real part. 
\end{prop}
\begin{proof} 
We show the invertibility of $G_4$ first. Let $v_2$ be a null vector of $G_4$. Take $v=[0 , v_2^\top ]^\top\in \IR^N $.  Then  $v^\top G v=v_4^\top G_4 v_2=0\ge \epsilon \|v_2\|^2$  implies $v_2=0$, i.e., the invertibility. 
Second, 
multiplying with $V_C^\top  G^{-1}$ on (\ref{sys}) yields one differential equation for $x_1$
 \beqq\label{eq15'}
B_{1,1} \frac{dx_1}{dt}+x_1=V_C^\top  G^{-1} V_C C_1 x_1+x_1=V_C^\top  G^{-1} u.
\eeqq
Let $H=G^{-1}$. With $V=[V_C, V_\cN]$, write $V^\top H V=\left(
\begin{array}{cc}
 H_1 & H_2 \\
H_3  & H_4  
\end{array}
\right).$  Since  $B_{1,1}=V_C^\top  B V_C=V_C^\top  G^{-1} V_C C_1=H_1 C_1$, we can calculate  one  explicit form for $H_1^{-1}$. Indeed,   $GH=I$ gives $H_3=G_4^{-1} G_3 H_1$ and $(G_1-G_2 G_4^{-1} G_3)H_1=$ the identity matrix. 
Likewise, $HG=I$ gives $H_1(G_1-G_2 G_4^{-1} G_3)=$ the identity matrix. Therefore,   $G_1-G_2 G_4^{-1} G_3$ is   $H_1^{-1}$, 
and thus the invertibility of $B_{1,1}$ is verified,  
\[
B_{1,1}^{-1}=C_1^{-1} (V_C^\top  G^{-1} V_C)^{-1}= C_1^{-1} (G_1-G_2 G_4^{-1} G_3).
\] 
Lastly, let $v$ be one eigenvector of $B_{1,1}$ corresponding to eigenvalue $\lambda$. Choosing  $U_\cR=V_\cR=V_C$, \[
\lambda  v= V_C^\top  B V_C v=V_C^\top  G^{-1} C V_C v=V_C^\top G^{-1} V_C C_1 v
\] 
implies
\[
\lambda   v^*C_1 v= (V_C C_1 v)^*   G^{-1} V_C C_1 v
\] 
and thus the positive real part is verified by  
\[
\Re(\lambda)   v^*C_1 v=\frac{1}{2} (G^{-1}V_C C_1 v)^*   (G^\top +G) (G^{-1} V_CC_1 v).
\]
 \end{proof}
%
%
%

With the above proposition, we can derive the solution to (\ref{eq15'}) as stated below.

\begin{prop}\label{case2}  Assume that $C, G$ satisfy  (\ref{assGC}).  Let $V:=[V_C, V_\cN]$ in (\ref{V_def}).
Let $B_{1,1}:= V_C^\top  G^{-1} C V_C$.
 Let $\widetilde u:=(B_{1,1})^{-1} V_C^\top  G^{-1} u$. Then  the projected  solution $x_\cR(t)$  is given by
\beqq\label{sol2}
x_{\cR}(t):=V_C x_1(t)= V_C  \left\{\exp(-t B_{1,1}^{-1}) V_C ^\top x(0)+\exp(-t B_{1,1}^{-1}) \int_0^t \exp(s B_{1,1}^{-1}) \widetilde u(s) \, ds\right\}.
\eeqq
In addition,  the projected  solution $x_\cN(t)$  is given by
\beqq\label{sol2'}
x_{\cN}(t):=V_\cN x_2(t)=V_\cN  (U_\cN^\top  G V_\cN)^{-1}U_\cN^\top (u(t)-G V_C x_1(t)).
\eeqq

\end{prop}

%
%
\begin{rem}
Suppose $G_4$ is invertible.
 Suppose $\widetilde u(s)$ is linear, i.e.,  with  some  vectors $\widetilde u(0), \widetilde u'(0)=\frac{d\widetilde u}{ds}(0)$,  we have\[ \widetilde u(s)=\widetilde u(0)+s \widetilde u'(0).\]
Then the second term in (\ref{sol2}) can be further simplified, i.e., 
\begin{eqnarray}
&&\exp(-t B_{1,1}^{-1}) \int_0^t \exp(s B_{1,1}^{-1}) \widetilde u(s) \, ds\\
&=&B_{1,1}\{ \widetilde u(t)-\exp(-t B_{1,1}^{-1}) \widetilde u(0)\}-B_{1,1}^2 (I-\exp(-t B_{1,1}^{-1})) \widetilde u'(0)
\\
&=&(-B_{1,1})\{ -I+\exp(-t B_{1,1}^{-1}) \} \widetilde u(0)+B_{1,1}^2 (-I+B_{1,1}^{-1} t+\exp(-t B_{1,1}^{-1})) \widetilde u'(0)\\
&=&t \varphi_1 (-t B_{1,1}^{-1}) \widetilde u(0)+t^2\varphi_2(-t B_{1,1}^{-1}) \widetilde u'(0).
\end{eqnarray}
Recall $\widetilde u(t)=(B_{1,1})^{-1}   V_C ^\top  G^{-1}   u(t)$. Thus, the  projected solution $V_C  V_C ^\top x(t)$ is given by 
\beqq\label{solx3}
 x_\cR (t)= \left\{ V_C  \exp(-t B_{1,1}^{-1}) V_C ^\top  x(0)+ t   V_C   B_{1,1}^{-1} \varphi_1 (-t B_{1,1}^{-1}) V_C ^\top  G^{-1}    u(0)+t^2  V_C   B_{1,1}^{-1}  \varphi_2(-t B_{1,1}^{-1}) 
V_C ^\top  G^{-1} u'(0)\right\}
\eeqq
\end{rem}

\begin{rem}
What happens if $G_4$ is not invertible?
This is one limitation of the decomposition  described in section~\ref{Exact_sol_2}: when $G_4$ is not invertible, then $B_{1,1}$ has   rank less than $n$ and $B$ can have generalized eigenvectors (in addition to null vectors) corresponding to eigenvalue $0$. 
      Non-invertibility of $G_4$ will lead to the  dimension decreases,   $\textrm{rank} (P_C G^{-1}C)<  \textrm{rank} (C)$, i.e., $V_\cR+V_\cN\neq \IR^N$.     
Actually, when  $G_4$ is not invertible, i.e.,  $G_4y=0$ for some nonzero vector $y$,    a zero eigenvalue of algebraic multiplicity for  $G^{-1} C$ is greater than its geometric multiplicity. The Jordan normal form of $B=G^{-1} C$ can have  eigenvalue with  has order $2$. 
More discussions can be found in
 Theorem 1 in \cite{IRA} and Theorem 2.7 in \cite{ERICSSON}. Further analysis on this issue   is beyond the scope of the current paper. 
\end{rem} 

\subsection{Krylov subspace approximation}\label{Krylov}

Since $G^{-1} C$ is well-defined, it is intuitive to 
apply  the  shift-and-invert Arnoldi iterations to compute  the requisite  matrix exponentials    in solving  (\ref{sys}) with  \textit{ singular } $C$. 
To compute $x_\cR(t)$ from   (\ref{sol2})  or (\ref{solx3})   for a large singular system in (\ref{sys}),
we shall design one $m$-dimensional Arnoldi algorithm to construct a low-dimensional rational Krylov subspace approximation
of the matrix exponential 
 of $B_{1,1}:=V_C^\top  G^{-1} C V_C$.

  Rational Krylov algorithms were originally developed for computing eigenvalues and eigenvectors of large matrices\cite{RUHE84}. 
 Unlike polynomial approximants,   rational best approximants of $\exp(-x)$ can converge geometrically in the domain $[0, \infty)$~\cite{CODY69}. 
Rational Krylov subspace method is a   very promising manner in 
computing  
 matrix exponentials $\phi_k(-t A)$ acting on a vector $v$,  
  when the numerical range of $A$ is located somewhere in the right half complex plane.
  Typically, the numerical range of the matrix $B_{1,1}$ does not completely lie in the right half plane. 
    In~\cite{AWESOME}, a new Arnoldi scheme with structured orthogonalization  is introduced to generate one stable Krylov subspace
 and to  compute  matrix exponentials.   The orthogonality is based on the positive semi-definite matrix $C$.
  The  orthogonality induced by the $C$ semi-inner product actually  plays a fundamental role in enforcing the numerical range of the operator in the right half plane under the assumption in (\ref{assGC}).

 \subsubsection{Shift-and-invert methods}

 
 \begin{rem} 
Fix  some parameter $\gamma>0$.  The shift-and-invert method approximates  $\phi_k(-t A) v$
 in the resolvent Krylov subspace, 
 \[
 span\{v, (\gamma I+A)^{-1} v, \ldots, (\gamma I+A)^{-(m-1)} v\}.
 \]
As one reference, we list the result for the nonsingular case.  Let $A=C^{-1} G$. The standard Arnoldi iterations are used to construct $(V_m, H_m)$ from 
\[
(C+\gamma G)^{-1} C V_m=V_m H_m +h_{m+1,m} v_{m+1} e_m^\top,
\] where columns of $V_m$ are a set of   orthogonal vectors of  $m$-dimensional Krylov subspace induced by  $(C+\gamma G)^{-1} C$ and $H_m$  satisfies 
\[
H_m =V_m^\top (C+\gamma G)^{-1} C V_m.
\] 
When $h_{m+1,m}=0$,  $(C+\gamma G)^{-1} C$ can be approximated by $V_m H_m V_m^\top$ and 
 then the matrix exponential can be approximated by  \[
\exp(-tA) v\approx \|v\| V_m \exp(t (I-H_m^{-1})/\gamma) e_1.
\]
\end{rem}

 \begin{definition}
 To estimate the eigen-structure of $G^{-1}C$ subject to $\cR$, we introduce a few matrices $S, S_{1,1}$ associated to $ B$,
 \begin{eqnarray}
&&
   S:= P_C (C+\gamma G)^{-1} C,\;    \widetilde S:=  (C+\gamma G)^{-1} C,\;  \label{S_def}
 \\
&&  S_{1,1}:=V_C ^\top  SV_C=V_C ^\top  \widetilde SV_C, \label{eq27}
\; \gamma>0. 
\end{eqnarray}
\end{definition}

Let $W_m:=[ w_1, w_2, \ldots, w_m]$ be one low-dimensional   subspace in the range of $P_C G^{-1} C$ and
$H_m$ be one upper Hessenburg matrix $H_m$
corresponding to the projection of  $P_C G^{-1} C$  on $W_m$,
 where   $\{ w_1, w_2, \ldots, w_m\} \in \IR^{N}$ with  $C$-orthogonality span
one Krylov  subspace from the operator $S$,
 \[
   span\{w_1, S w_1, S^2 w_1, \ldots S^{m-1} w_1\}=span\{w_1, w_2, \ldots, w_m\}.
   \]
The algorithm to generate $(W_m, H_m)$ is stated in Algorithm~\ref{algo_ortho_arnoldi_singular}.
Empirically we use the Arnoldi iterations in (\ref{eq29})  to compute 
  $\widetilde W_m$ and $H_m$
  instead. 
  Prop.~\ref{prop1.6} suggests computation  of the  approximate $x_a(t)$ in   (\ref{eq52}) is involved with 
   one single operation $P_C$.
     Since     $W_m$ is the  projection of $\widetilde W_m$ under $P_C$,  the upper Hessenburg matrix $H_m$ are identical.  Then    the matrix exponentials can be approximated by 
  (\ref{eq52}), where only one $P_C$ projection is applied.   
Observe that when  $h_{m+1,m}=0$ in  (\ref{eq28}),  then  we have $S=W_m H_m  W_m^\top C$, which suggests 
the  approximation  $W_m H_m  W_m^\top C$ of $S$.  The proof is straightforward, thus omitted.

   \begin{prop}\label{prop1.6}
   Consider the following two  $C$-orthogonal   Arnoldi iterations to generate $(W_m, H_m)$ and $(\widetilde W_m, \widetilde H_m)$ from $S$ and $\widetilde S$, respectively: 
    \begin{eqnarray}\label{cV}
&&
S W_m=W_m H_m+h_{m+1,m} w_{m+1} e_m^\top,\label{eq28}\\
&&
\widetilde S \widetilde W_m=\widetilde W_m \widetilde H_m+\widetilde h_{m+1,m} \widetilde w_{m+1} e_m^\top,\label{eq29}
\end{eqnarray}
where  columns of $W_m$ and $\widetilde W_m$ both form two sets of  $C$-orthonormal vectors \[ W_m=[w_1, w_2,\ldots, w_m], \widetilde W_m=[\widetilde w_1, \widetilde w_2,\ldots, \widetilde w_m], \, W_m^\top  C W_m=\widetilde W_m^\top  C \widetilde W_m=I.\]
\begin{itemize}
\item Suppose the first column of $W_m$ lies in   the range of $P_C G^{-1} C$. Then 
all columns of $W_m$  lie  in the range of $P_C G^{-1} C$.
\item Suppose the first column of $\widetilde W_m$ lies in the range of $G^{-1} C$. Then 
all columns of $\widetilde W_m$  lie in the range of $G^{-1} C$.
\item Suppose $(\widetilde W_m, \widetilde H_m)$ satisfies (\ref{eq29}). Let $W_m=P_C \widetilde W_m$ and $H_m=\widetilde H_m$.  Then $(W_m, H_m)$ satisfies (\ref{eq28}). 
\end{itemize}
\end{prop}
%
%
 The $C$-orthogonality together with the positive definite assumption of $G$ indicates the  passivity property of $H_m$ and the invertibility. This is also known as the stability condition\cite{IRA2}.

\begin{algorithm} [hbt]
	\caption{ {\bf An Arnoldi algorithm with explicit structured orthogonalization  and implicit regularization}\cite{AWESOME}}
	\label{algo_ortho_arnoldi_singular} 
	\small
	\KwIn { $C, G, k, \gamma,   w,  m$  }
	\KwOut {$H_m, W_m$}
	{   
		Set $w =P_C w$\;
		$ w_1 = \frac{w }{\lVert  w\rVert_{C}}$ where $\lVert  w\rVert_{C} = \sqrt{w^\top {C}w}$ and $w_1^T\mathcal{C}w_1 = 1$ \;
		\For {$j=1:m$}
		{
			Solve $(\gamma G+C) w =  C w_j$ and obtain $w$\; 
			\label{algo_construct_line}
			Set $w = P_C w$\;
			\For {$i = 1:j$}
			{
				$h_{i,j} =  w^\top C w_{i}$\;
				$  w = w - h_{i,j} w_{i}$\;
			}
			$h_{j+1,j} = \lVert  w \rVert _\mathcal{C}$\;
			$w_{j+1} = \frac{ w }{h_{j+1,j}} $\;
			\If{residual $<$ tolerance} {
				Results converge at dimension $m$\;
			}
		}
	}	
\end{algorithm}

 \begin{rem}[Passivity  property] Assume $G,C$  given in (\ref{assGC}).
 The advantage of the $C$-orthogonal iterations in (\ref{cV}) lies in the preservation of the passivity property of $H_m$, i.e., 
  all eigenvalues of  $H_m$ have non-negative  real components. In particular, with $G$ positive definite, we have the invertibility of $H_m$, which is crucial to the algorithm as well as the error analysis.    Indeed, since  observe that (\ref{cV}) implies \beqq\label{H_def1}
 W_m^\top C S W_m= W_m^\top C P_C (C+\gamma G)^{-1}C W_m=W_m^\top C (C+\gamma G)^{-1}C W_m=H_m.
 \eeqq
 Then  for each nonzero vector $x\in \IR^m$, with $y:=(C+\gamma G)^{-1} (CW_m x)\in \IR^N$, we have
 \[
 \langle x, H_m x\rangle =  (CW_m x)^\top  (C+\gamma G)^{-1} (CW_m x)= y^\top (C+\gamma G) y\ge  0.
 \]
\end{rem}

The following  shows the relation between  $B_{1,1}$ and $S_{1,1}$.
\begin{prop} \label{prop1.8}Suppose $G$ is postive definite.  Let $\gamma>0$, and introduce the function $g:\IC\to \IC$ and its inverse $g_1$,  \[ \lambda=g(\mu)=(1+\gamma\mu^{-1})^{-1},\mu=g_1(\lambda):=g^{-1}(\lambda)=((\lambda^{-1}-1)/\gamma)^{-1}.\]
Then 
\beqq\label{BS}
  B_{1,1}=g^{-1}( S_{1,1}),\; S_{1,1}=g(B_{1,1}).\eeqq
\end{prop}
\begin{proof} 
By Prop.~\ref{prop1.2}, $B_{1,1}$ is invertible.
Let $T:=V_C^\top  G^{-1} V_C$ and $C_1=V_C^\top  C V_C$. 
Then   \beqq
  B_{1,1}=V_C^\top  G^{-1}C V_C 
= T C_1,
\eeqq
\begin{eqnarray}
 S_{1,1}&=& V_C ^\top  (G^{-1}(C+\gamma G))^{-1}   G^{-1}C V_C\\
& =& 
V_C ^\top  (G^{-1}C+\gamma I))^{-1} V_C V_C^\top    G^{-1}C V_C\\
 &=& (T C_1 +\gamma I)^{-1} TC_1 
 =g( B_{1,1}).
\end{eqnarray}
\end{proof}

Introduce a few notations.
  Let $g,g_1$ be given in Prop.~\ref{prop1.8} and \beqq
  \label{def_f} f(\lambda):=\varphi_0(-t (g^{-1}(\lambda))^{-1})=\varphi_0(-t g_1(\lambda)^{-1} )
  \eeqq
  and let \beqq\label{f1}
f_k(\lambda):= g_1(\lambda)^{-1} \varphi_k(-t g_1(\lambda)^{-1})
 \textrm{ for $k=1,2$.}
\eeqq
Now, we are ready to state 
 one approximation  $x_a(t)$ for  $x_\cR (t)$ in (\ref{solx3}). The error analysis will be given in next section.
  \begin{thm}
Let $(\widetilde W_m, H_m)$ and $( W_m, H_m)$ be generated from Arnoldi iterations with respect to $\widetilde S$ and $S$ in Prop.~\ref{prop1.6}.
Let 
\begin{eqnarray}\label{x_apr}
&& x_a(t):=W_m \left\{ f(  H_m) W_m^\top  C  x(0)+ 
  t   f_1 (H_m) W_m^\top  C u(0)+t^2   f_2( H_m) 
W_m^\top C u'(0) \right\}\\
&=& P_C \widetilde W_m \left\{ f(  H_m) \widetilde W_m^\top  C  x(0)+ 
  t   f_1 (H_m) \widetilde W_m^\top  C u(0)+t^2   f_2( H_m) 
\widetilde W_m^\top C u'(0) \right\}\label{x_apr_1}\label{eq52}
\end{eqnarray}
Suppose $x(0), u(0)$ and $u'(0)$ all lie in the range of $W_m$ and $h_{m+1,m}=0$.
Then $x_a(t)=x_\cR(t)$.
  \end{thm}
  \begin{proof}
  Write $x_\cR (t)$   in (\ref{solx3}) as follows, 
 \[  x_\cR (t):=z_1(t)+z_2(t)+z_3(t).\]
     The first term in  (\ref{solx3}) gives
\begin{eqnarray}
&&z_1(t)= V_C \exp(-t B_{1,1}^{-1}) V_C ^\top x(0)\label{eq20}
\\
&=& V_C \exp(-t  \{ g^{-1}(S_{1,1})\}^{-1}) V_C ^\top x(0)=V_C f(S_{1,1})V_C ^\top x(0).\label{eq_21}
\end{eqnarray}
The approximation of (\ref{eq_21}) is computed  as follows.
From 
\[
 S_{1,1}\approx V_C ^\top  W_m H_m W_m^\top C V_C,\]
 and $V_C V_C^\top  W_m=W_m$, 
we have  \beqq\label{eq40}
(S_{1,1})^k \approx V_C ^\top  W_m H_m^k W_m^\top C V_C .\eeqq
Since columns of $W_m$ lie in $V_C$, then with $C$-orthogonality, 
 (\ref{eq_21}) yields
\begin{eqnarray}\label{eq45}
z_1(t)\approx 
 W_m f(  H_m) W_m^\top  CV_C V_C^\top   x(0)\approx W_m f(  H_m) W_m^\top  C  x(0).
\label{eq24}\end{eqnarray}
\footnote{In the case of $h_{m+1,m}=0$, the equalities in (\ref{eq40})
hold and thus 
 the equalities in (\ref{eq45}) hold.} 
For the remaining terms $z_2(t),z_3(t)$ of  (\ref{solx3}), we have \[
 V_C \varphi_0(-t B_{1,1}^{-1}) V_C ^\top \widetilde u(0)=V_C \exp(-t B_{1,1}^{-1}) V_C ^\top \widetilde u(0) \approx  W_m f(  H_m) W_m^\top C \widetilde u(0). 
 \]
Likewise, since $(B_{1,1})^{-1}=(g^{-1}(S_{1,1}))^{-1}=g_1(S_{1,1})$, then 
\[
 V_C  B_{1,1}^{-1} \varphi_k(-t B_{1,1}^{-1}) V_C ^\top \widetilde u'(0)= V_C  g_1(S_{1,1})^{-1} \varphi_k(-t g_1(S_{1,1})^{-1}) V_C ^\top \widetilde u'(0) \approx  W_m f_k(  H_m) W_m^\top C \widetilde u'(0). 
 \]
In summary,   we have (\ref{x_apr}) and (\ref{x_apr_1}) by Prop.~\ref{prop1.6}. 
\end{proof}
\begin{rem}[Complete solutions $x(t)$]With (\ref{x_apr}),  we can compute the complete solution $x_\cR(t)+x_\cN(t)$. From (\ref{sys}), 
\beqq \label{x_complete}x(t)=x_\cR(t)+x_\cN(t)=G^{-1}u(t)-G^{-1}C \frac{dx_\cR(t)}{dt},\eeqq
where 
\begin{eqnarray}
 && \frac{dx_\cR(t)}{dt}=
 W_m  \{ 
 -g_1(H_m)^{-1} \exp(-t g_1(H_m)^{-1}) W_m^\top  C  x(0)
\\&& + 
 g_1(H_m)^{-1} \exp(-t g_1(H_m)^{-1}) W_m^\top  C G^{-1}u(0)
 +
 (I- \exp(-t g_1(H_m)^{-1}))
W_m^\top C G^{-1}u'(0) \}\\
&=&W_m  \{ 
 g_1(H_m)^{-1} \exp(-t g_1(H_m)^{-1}) W_m^\top  C  (-x(0)+G^{-1}u(0))
\\&& 
 +
 (I- \exp(-t g_1(H_m)^{-1}))
W_m^\top C G^{-1} u'(0) \}.
\end{eqnarray}
\end{rem}

\begin{rem}  How to choose the initial vectors for the Arnoldi iterations?
Suppose  $x(0)$, $u(0)$ and $u'(0)$ lying in $\cR$. Then it is typical  to choose them as  the initial vector of the corresponding  Arnoldi iterations with a proper $C$-normalization, i.e., the first column of $\widetilde W_m$ is the normalized   vector $w/\langle w, Cw\rangle^{1/2}$. 
Note that when $(\widetilde W_m^{(0)}, H_m^{(0)})$ is generated from the $C$-orthogonal Arnoldi iterations with the initial vector $x_0$ in $\cR$, 
 the first term of $x_a(t)$ in (\ref{x_apr}) becomes
$  \beta_0 P_C \widetilde W_m^{(0)}  f(  H_m ^{(0)}) e_1 $, where $ \beta_0=\|x_0\|_C$.
Empirically, one can collect all the exponential terms as one  matrix-exponential-and-vector product  (either $\varphi_0$, $\varphi_1$ or $\varphi_2$) and construct only one pair of $(W, H)$ to conduct the computation, as considered in~\cite{AWESOME}    
%
%
\end{rem}

\section{Error analysis} 

\subsection{$C$-numerical range}
  The numerical range (or called field of values)~\cite{
Charles, Crou07, Bern09}, which is the range of Rayleigh quotient,  
 is one fundamental quantity in the error analysis of matrix exponential computation. To establish the convergence,  for a square matrix $A\in \IC^{N\times N } $ of the form $A=KC$ with some matrix $K\in \IR^{N\times N}$,  we introduce   the $C$-numerical range 
 \beqq\label{FC}
 \cF_C (A)=\{x^* CA x: x\in \IC^N, \|x\|_C:=\sqrt{x^*Cx}\le 1 \},
 \eeqq
which is one generalization of the standard  numerical range 
 \[
 \cF(A)=\{x^* A x: x\in \IC^N, \|x\|\le 1\}.
 \]
 Here $A$ could be the  matrix $B$ in (\ref{Bdef}) or $S$ in (\ref{S_def}). 
Clearly, the set $\cF_C(A)$ in (\ref{FC}) only depending on those vectors $x$ in the range $C$.

  \begin{definition} The set of a disk with center $c_1\in \IC$ and radius $\rho_1>0$ is denoted by $\cD(c_1, \rho_1)\subset \IC$.
 \end{definition}

 Due to possible non-symmetric structure in $G$, numerical range $\cF_C (B)$ is not a line-segment on the real axis. 
The smallest  disk  covering  $\cF_C(B)$ is introduced to
quantize the spectrum of $B=G^{-1} C$. 
For $G,C$ in (\ref{assGC}), let $C=V_C C_1V_C^\top$ be the eigenvector decomposition. 
 Note that eigenvalues of $B_{1,1}$ all lying  in the right half plane from Prop.~\ref{prop1.2} does not implies that $\cF(B_{1,1}) $ lies in the right half plane. As an alternate,  the  $C$-numerical range  $\cF_C(B)$  always lies in the right half plane.

\begin{prop}  
Let $A$ be in the form $A=K C$ for some matrix $K\in \IR^{N\times N}$. Then  both $\cF(A)$ and $\cF_C( A)$ contain
all nonzero eigenvalues of $A$. In addition, if $K$ is positive semi-definite, then $\cF_C(A)$ lies in the right half plane.
\end{prop}
\begin{proof}
Let $x$ be an nonzero eigenvector of $A$ corresponding to  nonzero eigenvalue $\lambda$. 
Then $Ax=\lambda x$ and the first statement is given by 
\[
\lambda=\frac{x^* Ax }{x^* x}=\frac{x^* C Ax }{x^* C x}.
\]
In addition, if $K+K^\top\succeq 0$, then
 \[
 \frac{x^* C Ax }{x^* C x}= \frac{x^* C KC x }{x^* C x}=
 \frac{x^* C (K+K^\top) Cx }{2 x^* C x} \; \forall x
 \]
 have a nonnegative real component. 

\end{proof}

 Here are a few properties of  $\cF_C(B)$ if $G$ is positive definite.
 
   \begin{prop}\label{FCB} Suppose that (\ref{assGC}) holds for $G,C$. Let $H=G^{-1}$. In addition, $(H+H^\top)/2$ is positive definite with eigenvalues in $[\xi_1, \xi_2]$ with $\xi_1>0$, $(H-H^\top)/2$ has eigenvalues in $[-i\xi_3, i\xi_3]$,  
 and $C$ is positive semi-definite with eigenvalues in $\{0\}\cup [\xi_4, \xi_5]$, $\xi_4>0$. Then $\cF_C(B)$ lies in $\cD(c_1, \rho_1)$ with $c_1>\rho_1$. Here $c_1,\rho_1$ only depend on these parameters $\xi_1,\xi_2,\xi_3,\xi_4,\xi_5$ of $C,G$.
 \end{prop}
 \begin{proof}
 Note that    the $C$-numerical range of $B$ can be expressed  by 
\beqq\label{FB}
\cF_C(B)=\{ \frac{x^* C G^{-1} C x}{x^* C x}: x\in \IC^N, Cx\neq 0 \}=\{ z^* D_C^{1/2}V_C^\top G^{-1} V_C  D_C^{1/2} z: \|z\|\le 1, z\in \IC^n\}.
\eeqq
From $G^{-1}=(H+H^\top)/2+(H-H^\top)/2$, then $\cF_C(B)$ lies within  a box region in the right half plane,
\[ 0<\xi_1\xi_4 \le \Re(\cF_C(B))\le \xi_2 \xi_5,\;  -\xi_3\xi_5 \le \Im(\cF_C(B))\le \xi_3 \xi_5, 
\] 
where  equalities can hold only if $z$ is a pure real vector or a pure imaginary vector.  
Thus, we can find  some $c_1>0, \rho_1>0$ with $c_1-\rho_1>0$ , 
such that  $\cF_C(B)\subset \cD(c_1, \rho_1)$. 
Choose \[\rho_1:=\sqrt{
(\max( c_1-\xi_1\xi_4, \xi_2\xi_5 -c_1))^2+(\xi_3\xi_5)^2
}.
\]
Note that $c_1^2\ge \rho_1^2$ holds if and only if 
\[
c_1\ge \max\{ (2\xi_1\xi_4)^{-1} \{\xi_1^2\xi_4^2+\xi_3^2\xi_5^2\}, (2\xi_2\xi_5)^{-1} \{\xi_2^2\xi_5^2+\xi_3^2\xi_5^2\}\}
\]
Hence, with a sufficient large value $c_1$, the disk $\cD(c_1,\rho_1)$ containing $\cF_C(B)$ lies in the right half plane. 
 \end{proof}

In general $B$ is not normal.  The following proposition and remark   exhibit   the dependence of $\cF_C(S)$ and $\cF(H_m)$ on $\cF_C(B)$. As long as $\cF_C(B)$ lies in the right half plane, $\cF(H_m)$ does as well. 
The following  function $g$ 
which is one M{\"o}bius transformation maps generalized circles to generalized circles, which actually lie within $\cD(1/2,1/2)$.

\begin{prop}\label{Bound1}Let $\gamma>0$ and $\lambda=g(\mu)=(1+\gamma\mu^{-1})^{-1}$, which maps $\mu \in \cF_C(B)$  to $\lambda\in \cF_C(S)$ by (\ref{BS}). 
Suppose (\ref{assGC}). Then Prop.~\ref{FCB} indicates that  $\cF_C(B)$ lies in the right half plane,
 \beqq \label{ass1}\cF_C(B)\subset \cD(c_1,\rho_1) \textrm{ with some } c_1, \rho_1 \in \IR. 
 \eeqq 
 Let $\mu_1: =c_1-\rho_1>0, \textrm{ and } \mu_2:=c_1+\rho_1$. 
Then $\cF_C(S)\subset \cD(c_0, \rho_0)$, where $c_0=(g(\mu_1)+g(\mu_2))/2$, $\rho_0=(g(\mu_2)-g(\mu_1))/2$.
Note that since $\mu_1\ge 0, \mu_2\ge 0$, then $g(\mu_2)\le 1$ and $g(\mu_1)\ge 0$. Thus, $\cF(S)\subset \cD(1/2,1/2)$.
\end{prop}

\begin{proof} 
Consider the mapping theorem by Berger-Stampfli(1967)\cite{Berger}. 
Let \beqq\label{TB}
T=(B-c_1I)/\rho_1,
\eeqq i.e., $B=c_1 I+\rho_1 T$ where $c_1=(\mu_1+\mu_2)/2$, $\rho_1=(\mu_2-\mu_1)/2$. 
Then by  (\ref{ass1}),  $|\cF_C (T)|\le 1$.

Choose one analytic function $f$ on $z\in \cD(0,1)\to \cD(0,1)$,   \[
f(z)=\frac{g(\rho_1 z+c_1)-c_0}{\rho_0}. 
\]
Since $g$ is a function mapping a circle with centre at the real axis to another  circle with  center at the real axis,  by definition of $c_0, c_1, \rho_1$,  $|f(z)|\le 1$ for all $|z|\le 1$.
Clearly, $f(z)$ is analytic in $|z|<1$ and continuous on the boundary.
By the theorem in~\cite{Berger}, $\cF_C(f(B))$ also lies in $\cD(0,1)$. Thus with (\ref{TB}),
\[\cF_C(S)=\cF_C(g(B))=
c_0+\rho_0\cF_C \left(\frac{g(\rho_1 T+c_1I)-c_0I }{\rho_0}\right) 
\]
lies in the disk $\cD(c_0, \rho_0)$, i.e., with center $c_0$ and radius $\rho_0$.

\end{proof}
\begin{rem} 
The passivity property of the system indicates $\cF(H_m)$ in $\cD(c_0,\rho_0)$. Indeed, 
from (\ref{H_def1}) and $W_m^\top CW_m=I$,  the numerical range of $H_m$ lies inside the $C$-numerical range, \beqq\label{HS}
\cF(H_m)\subset \cF_C(S)\eeqq
 according to the definition of $\cF$ and $\cF_C$.
\end{rem}

To establish the convergence, we need the following results. 
The coming  result relates  the spectral norm to  the radius of its  numerical range. 

\begin{prop}\label{bound2}
Let $A\in \IR^{N\times N}$ in the form of $KC$ with $K\in \IR^{N\times N}$ and $\cF_C(A)$ lie in $\cD(0, \rho)$. Then \[ 
\|A\|_C:=\sup_v\{ \|A v\|_C/\|v\|_C\}\le 2\rho.\]
\end{prop}
\begin{proof} Since $C$ is unitary diagonalizable,  $C=V_C C_1V_C^\top $, the matrix square root   $C^{1/2}$ is given by  $C^{1/2}=V_C D_C^{1/2} V_C^\top $. 
For any $v\in \IR^N$ with $\|v\|_C=1$, we have \begin{eqnarray}
&&\frac{\|(K + K^\top )Cx \|_C}{\|x\|_C}=\frac{\| C^{1/2}(K+K^\top)C^{1/2}C^{1/2} x \|}{\|C^{1/2} x\|}=
\| C^{1/2}(K+K^\top)C^{1/2} \|\\
&=&\max_x \Re(x^*(CK C+C K^\top C) x)=2\max_x  \Re(x^* C K C x)\le 2\rho.\end{eqnarray}
Likewise, 
$\|v\|_C^{-1}\|(K-K^\top )Cv \|_C=
\|C^{1/2}(K-K^\top )C^{1/2} \|
=\max_x \Im(x^*(CK C-CK^\top C) x)\le 2\rho.$
The sum of the above two inequalities gives
  $
  \|K Cv \|_C/\|v\|_C
  \le 2\rho.$
\end{proof}

The following inequality induces   numerical range $\cF_C(A)$ in  estimating error bounds for (\ref{eq69}). 
\begin{prop}\label{bound3}
Let $\Gamma$ be a set in $\IC$ and $d(\Gamma, \cF_C(A))$ be the shortest distance between $\Gamma$ and $\cF_C(A)$. Then 
\[\min_{\lambda\in \Gamma} \|(\lambda I-A)^{-1}\|_C \le d(\Gamma, \cF_C(A))^{-1}.\]
\end{prop}
\begin{proof}
Let $u=(\lambda I-A)^{-1} v\in \IC^n$. Then for each $\lambda\in \Gamma$,
\[
 d(\Gamma, \cF_C(A))\le \frac{|\langle u, C(\lambda I-A )u \rangle |}{\|u\|_C^2}=\|u\|_C^{-2}|\langle  u, v\rangle_C |\le \|u\|_C^{-1}\cdot  \|v\|_C
\]
Hence, for each vector $v$, we have 
\[
\frac{\| (\lambda I -A)^{-1} v\|_C}{\|v\|_C}=\frac{\|u\|_C}{\|v\|_C}\le d(\Gamma, \cF_C(A))^{-1},
\]
which completes the proof.
\end{proof}

\subsection{A posterior error bounds (residual)}\label{Posterior}
A posteriori error
estimates are crucial in practical computations, e.g.,  determining the dimension $m$ of the Krylov space for  (\ref{x_apr}) or the time span used 
in the matrix exponential.
In the following, we apply the residual arguments  in~\cite{Botchev} to estimate errors of (\ref{x_apr})  in the case of  $h_{m+1,m}\neq 0$. 
Here we focus on  the  term involving with $\phi_0$ in (\ref{x_apr}) for the sake of simplicity.

   \begin{prop}
 Let $y_m(t)$ be the first term of 
 the approximation of $x_\cR(t)$ in (\ref{solx3}), i.e., 
 the first term in (\ref{x_apr}), 
  \[y_m(t):=W_m \varphi_0(-tg_1(H_m) )W_m^\top C x(0)\in \cR.\]
Denote the residual function  by $r_m(t)$
\[
 r_m(t):=P_C G^{-1}C 
\frac{dy_m}{dt}+   y_m.\]
Then 
\beqq\label{residual}
r_m(t)=- \beta(t) P_C \{G^{-1}( C+\gamma G)  w_{m+1}\},
\eeqq
where $\beta(t)$ is a scalar, independent of whether $P_C$ is applied or not, 
 \begin{eqnarray}
&&\label{beta_m} \beta(t):=h_{m+1,m}\gamma^{-1} e_m^\top H_m^{-1}  \varphi_0(-tg_1(H_m) )W_m^\top C x(0)\\
&=&h_{m+1,m}\gamma^{-1} e_m^\top H_m^{-1}  \varphi_0(-tg_1(H_m) )\widetilde W_m^\top C x(0) \in \IR.
\end{eqnarray}
\end{prop}

\begin{proof}  
Let $y(t)$ denote  the corresponding term of the exact solution in (\ref{solx3}),   $y(t):=V_C  \exp(-tB_{1,1}^{-1}) V_C ^\top x(0)$. Note that    $y(t)$ satisfies 
\beqq\label{xt} P_CG^{-1} C
\frac{dy}{dt}+  y=V_C B_{1,1} V_C ^\top
\frac{dy}{dt}+  y=(-P_C+V_C V_C^\top  ) G^{-1} C V_C B_{1,1}^{-1}  \exp(-tB_{1,1}^{-1}) V_C ^\top x(0)=0.
\eeqq
Since 
$V_C B_{1,1}V_C^\top  W_m=V_C V_C^\top  G^{-1} C V_C V_C^\top  W_m=P_C G^{-1} C W_m,$
and from the definition of $g_1$,   $(W_m, H_m)$ satisfies    (\ref{eq28}), 
then 
we have
\begin{eqnarray}
&&\label{eq47} V_C  B_{1,1} V_C ^\top W_m g_1(H_m)=P_CG^{-1} C W_m (H_m^{-1}-I)\gamma^{-1}
\\
&=&P_CG^{-1}\{
G W_m +\gamma^{-1}(C+\gamma G) h_{m+1,m} w_{m+1} e_m^\top H_m^{-1}
\}.
\end{eqnarray}
Then
 \begin{eqnarray}\label{xtm}
 r_m(t)&=&P_C G^{-1}C 
\frac{dy_m}{dt}+   y_m=V_C  B_{1,1} V_C ^\top 
\frac{dy_m}{dt}+ y_m\\
&=&
\left\{-V_C  B_{1,1} V_C ^\top W_m g_1(H_m) + W_m \right\} \varphi_0(-tg_1(H_m) )W_m^\top C x(0)\\
&=&- \beta(t) P_C \{ G^{-1}( C+\gamma G)  w_{m+1}\},\label{rtm}
\end{eqnarray}
where we used (\ref{eq47}) to get the last equality.
\end{proof}

The following computation  provides one  connection  from  the residual estimate giving in (\ref{residual}) to the error estimate  under the assumption in (\ref{omega}). One major tool is  that by eigenvector decomposition of \[C_1^{1/2}B_{1,1}^{-1}C_1^{-1/2}=C_1^{-1/2} (V_C^\top G^{-1} V_C)^{-1} C_1^{-1/2}=XDX^{-1},\] 
there exist  $K>0$ and $\omega> 0$ depending on $B_{1,1}$,
 \beqq\label{omega}
\|\exp(-tB_{1,1}^{-1})\|_{C_1} \le K \exp(-t \omega).
\eeqq
 Here introduce  $C_1$-norm 
 \[
 \|x\|_{C_1}=\Re(x^* C_1 x)^{1/2}=\|C_1^{1/2} T x\|,\; 
  \|T \|_{C_1}:=\max_{x\neq 0}\frac{\Re((Tx)^* C_1 Tx)^{1/2}}{x^* C_1 x}=\max_{x\neq 0}\frac{\| Tx\|_{C_1}}{\| x\|_{C_1}}
  \] for vectors $x\in \IC^n$ and $T\in \IR^{n\times n}$. 
For instance,
 one can choose $K=\|X\| \|X^{-1}\|$ and choose  $\omega$  to be  the largest eigenvalue of \[  C_1^{1/2} (B_{1,1}^{-1}+(B_{1,1}^{-1})^\top)C_1^{-1/2}
/2=C_1^{-1/2}\frac{ (V_C^\top G^{-1} V_C)^{-1}+(V_C^\top (G^{-1})^\top V_C)^{-1}}{2} C_1^{-1/2}.\]

\begin{thm} 
Suppose $C, G$ satisfy (\ref{assGC}). 
Let $r_m(t)$ and $\beta$ be defined in (\ref{rtm}) and (\ref{beta_m}).
Let \[\epsilon_m(t)=y_m(t)-y(t)=
W_m \varphi_0(-tg_1(H_m) )W_m^\top C x(0)-V_C  \exp(-tB_{1,1}^{-1}) V_C ^\top x(0).
\]
Then (\ref{omega}) holds
for  some constants  $\omega, K$, depending on $B_{1,1}^{-1}$,  and
\begin{eqnarray}\label{bound_phi}
&& \| P_C \epsilon_m(t)\|_C \le K t \varphi_1(-t \omega) \max_{0\le s\le t}\| B_{1,1}^{-1} V_C ^\top   r_m(s)\|_{C_1}
\\
&\le& K t \varphi_1(-t \omega)  \cdot  \| (I+\gamma B_{1,1}^{-1}) V_C ^\top   w_{m+1}\|_{C_1} \cdot 
   \sup_{\theta\in [0,1]} \|\beta(t\theta )\|.\label{eq56}
\end{eqnarray}

\end{thm}

\begin{proof} 
By (\ref{assGC}) and Prop.~\ref{FCB}, $\cF_C(B)$  lies in the right half plane. This establishes  the existence of $K$ and $\omega$ in (\ref{omega}). 
From (\ref{xt}) and (\ref{xtm}), we can establish one equation between 
the error vector $\epsilon_m (t)= y_m(t)-y(t)$ and the residual vector $r_m(t)$,
\[ V_C^\top   G^{-1} C V_C V_C^\top   \frac{d\epsilon_m(t)}{dt}+ V_C^\top   \epsilon_m(t)=V_C^\top  r_m(t).
\]
Thus variation of constants formula gives 
\begin{eqnarray}
&& V_C^\top    \epsilon_m(t)=
 V_C^\top   \epsilon_m(t)=\int_0^t  \exp(-(t-s) B_{1,1}^{-1} ) B_{1,1}^{-1} V_C ^\top  r_m(s) ds
\\
&=&
\int_0^1  \exp(-t(1-\theta) B_{1,1}^{-1} ) B_{1,1}^{-1} V_C ^\top  r_m(t\theta) d\theta.
\end{eqnarray}
Examine  the definition of $\varphi_1$,
\[
\varphi_k(-t B_{1,1}^{-1})=\int_0^1 \exp(-(1-\theta)t B_{1,1}^{-1}) \frac{\theta^{k-1}}{(k-1)!} d\theta, k\ge 1, \textrm{ which yields } \|\varphi_1(-t B_{1,1}^{-1})\|_{C_1}\le K \varphi_1(-t \omega). 
\]
Hence, we have the upper bound for the error vector, 
\begin{eqnarray}
&&\| V_C   \epsilon_m(t)\|_{C_1}\le 
\int_0^1\|  \exp(-t(1-\theta) B_{1,1}^{-1} ) \|_{C_1}  d\theta\,  \{ \sup_{\theta\in [0,1]} \|  B_{1,1}^{-1} V_C ^\top  r_m(t\theta)\|_{C_1}\}.
\\
&\le& K
\int_0^1  \exp(-t(1-\theta) \omega )  d\theta\,  \{  \sup_{\theta\in [0,1]} \| B_{1,1}^{-1} V_C ^\top r_m(t\theta)\|_{C_1}\}\\
&=& K
\varphi_1(-t \omega ) \,  \{  \sup_{\theta\in [0,1]} \| B_{1,1}^{-1} V_C ^\top  r_m(t\theta)\|_{C_1}\}.
\end{eqnarray}
The proof is completed by using  (\ref{residual}),
 \begin{eqnarray}
&&  \| B_{1,1}^{-1} V_C ^\top  r_m(t\theta)\|_{C_1}\le 
   \| B_{1,1}^{-1} V_C ^\top  P_C (G^{-1} C V_C V_C^\top  +\gamma I) w_{m+1}\|_{C_1} 
    \sup_{\theta\in [0,1]} \|\beta(t\theta )\|
  \\
  & =&  \| (I+\gamma B_{1,1}^{-1}) V_C ^\top   w_{m+1}\|_{C_1} \cdot 
   \sup_{\theta\in [0,1]} \|\beta(t\theta )\|.
 \end{eqnarray}

\end{proof}


\subsection{Error bound inequality}
The previous residual  analysis does not  explicitly reveal  the error convergence behavior as the Krylov dimension increases. In the following, we shall    establish one upper bound depending on   time span $t$, dimension  $m$ and $\gamma$ to show the convergence in computing
the matrix exponentials. 
Literatures\cite{Saad92}\cite{Hoch} show that   the error of $m$-dimensional approximations of matrix exponentials could decay at least linearly (super-linearly), as the Krylov dimension increases.  We shall 
examine  the case, where 
  the $C$-orthogonality Arnoldi iterations are employed to implement  the shift-and-invert method.  

 The following error bound  shows the effectiveness of  $C$-orthogonality Arnoldi algorithms in solving $x_\cR(t)$ of (\ref{sys}) under (\ref{assGC}).
  From (\ref{BS}), (\ref{solx3}) and (\ref{x_apr}), the quality of  $x_a$ in (\ref{x_apr}) can be analyzed in the following inequality,
\begin{eqnarray}
\label{x_bd0}
&&\| x_\cR (t)-x_a(t)\|_C \le 
  \|\{V_C  f(S_{1,1}) V_C ^\top  -W_m^{(0)} f(H_m^{(0)}) {W_m^{(0)} }^\top  C \} x(0) \|_C \\
  &+&
   \|\{V_C f_1(S_{1,1}) V_C ^\top  -W_m^{(1)}  f_1(H_m^{(1)} ){W_m^{(1)} }^\top  C \}u(0) \|_C \label{eq103'0}\\
   &+&
  \|\{V_C f_2(S_{1,1}) V_C ^\top  -W_m^{(2)}  f_2(H_m^{(2)} ){W_m^{(2)} }^\top  C \}u'(0) \|_C. \label{eq104'0}\end{eqnarray}

  \subsubsection{Convergence}
Suppose $G$ is only \textit{positive semi-definite}. 
The following Theorem~\ref{Thm3.1} is one error bound for $\varphi_l$ functions for $l\ge 1$, (from Theorem 5.9\cite{Tanja14}). 
Since the analysis cannot be used in the  $\varphi_0$-case, we consider $\varphi_1$ for the $x(0)$ term, i.e., 
with  $\varphi_0(-x)=(-x)\varphi_1(-x)+1$, we have
\[
 f(S_{1,1})= (-S_{1,1}) (f_1(S_{1,1}))+I,
\]
which gives the $\varphi_1$-computation for $f(S_{1,1})$, 
\begin{eqnarray}
&& V_C f(S_{1,1}) V_C^\top x(0)= (-V_CS_{1,1} V_C^\top ) V_C (f_1(S_{1,1})) V_C^\top x(0)+x(0),
\\
&\approx & (-V_CS_{1,1} V_C^\top ) ( W_m f_1(H_m) W_m^\top C)  x(0)+x(0).
\end{eqnarray}
Hence, 
\begin{eqnarray}
\label{x_bd}
&&\| x_\cR (t)-x_a(t)\|_C \le 
  \|\{(-S)\{ V_C f_1(S_{1,1}) V_C ^\top  - W_m^{(0)} f(H_m^{(0)}) {W_m^{(0)} }^\top  C \} x(0) \|_C \\
  &+&
   \|\{V_C f_1(S_{1,1}) V_C ^\top  -W_m^{(1)}  f_1(H_m^{(1)} ){W_m^{(1)} }^\top  C \}u(0) \|_C \label{eq103'}\\
   &+&
  \|\{V_C f_2(S_{1,1}) V_C ^\top  -W_m^{(2)}  f_2(H_m^{(2)} ){W_m^{(2)} }^\top  C \}u'(0) \|_C. \label{eq104'}\end{eqnarray}

Applying this theorem to~(\ref{x_bd}) gives Theorem~\ref{Thm3}, which describes the convergence to $x_\cR(t)$ in the positive semi-definite case. The convergence in $m$ is at least sub-linear. 
 \begin{thm}\label{Thm3.1}
 Let $A$ satisfy $\cF(A)\subseteq \IC_0^-$ and let $P_m=V_mV_m^\top$ be the orthogonal projection onto the shift-and-invert Krylov subspace
 $Q_m(A,v)$. For the restriction $A_m=P_m A P_m$ of $A$ to $Q_m(A,v)$, 
  we have the error bound
 \[
 \|\varphi_l(A) v-\varphi_l(A_m) v\|\le \frac{C(l,\gamma)}{m^{l/2}} \|v\|, \; l\ge 1. 
 \]
 \end{thm}

\begin{thm}\label{Thm3}
Suppose that $C,G$ are positive semi-definite and $C$ is symmetric.  Then $ \cF_C(G^{-1} C)$ lies in the right half complex plane. 
Replacement of the $x(0)$-term $W_m^{(0)} f(H_m^{(0)}) {W_m^{(0)} }^\top  C \} x(0)$ of $x_a(t)$  in (\ref{x_apr}) with  $ - S( W_m f_1(H_m) W_m^\top C)  x(0)+x(0)$.   
Then  \beqq\label{error0}
\| x_\cR (t)-x_a(t)\|_C \le   
 \frac{C(1,\gamma )}{m^{1/2}}  \|S\|_C  \|x(0)\|_C+  \frac{C(1,\gamma )}{m^{1/2}}   \|u(0)\|_C+ \frac{C(2,\gamma )}{m^{2/2}}  \|u'(0)\|_C.
\eeqq

\end{thm}
  
\begin{proof}
We shall verify the conditions stated in Theorem~\ref{Thm3.1}.
 Let \beqq
 \label{Adef}
 A:=-C_1^{1/2} B_{1,1}^{-1} C_1^{-1/2}.
 \eeqq
  Since $V_C^\top G^{-1} V_C$ is postive semi-definite, then  the positive definite condition on $G$ together with the calculation  \[
-A= C_1^{1/2} B_{1,1}^{-1} C_1^{-1/2}=C_1^{1/2} (V_C^\top G^{-1} C V_C)^{-1} C_1^{-1/2}=C_1^{-1/2} (V_C^\top G^{-1} V_C)^{-1}  C_1^{-1/2}.
 \]
implies that  the numerical range $\cF(-A)$ lies in the right half complex plane.
Let $Q_m(A,v)$ be  the shift-and-invert Krylov subspace
\[
Q_m(A,v)=span\{v, ( I-\gamma A)^{-1} v, \ldots, ( I-\gamma A)^{-(m-1)} v\}.
\]
Note that  the definition of $S_{1,1}$ gives  \beqq\label{SBeq}
 S_{1,1}=V_C^\top (C+\gamma G)^{-1} CV_C =(I+\gamma B_{1,1}^{-1})^{-1},
 \eeqq
 and 
 \[
 A=-\gamma^{-1}  C_1^{1/2} (S_{1,1}^{-1}-I)C_1^{-1/2}. 
 \]
 From (\ref{Adef}), we have
\[
(I-\gamma A)^{-1}=C_1^{1/2} ( I+ \gamma B_{1,1}^{-1})^{-1} C_1^{-1/2}=C_1^{1/2} S_{1,1} C_1^{-1/2}.
\] 
Thus, the subspace $Q_m(A,v)$ is  actually  the Krylov subspace $K_m(C_1^{1/2} S_{1,1} C_1^{-1/2} , v)$, i.e., 
\[
Q_m(A,v)=span\{v, C_1^{1/2} S_{1,1} C_1^{-1/2} v, \ldots, C_1^{1/2} S_{1,1}^{m-1} C_1^{-1/2} v\}.
\]
Let $V_m$ consist of orthogonal basis vectors in $K_m(C_1^{1/2} S_{1,1} C_1^{-1/2} , v)$. Then we have Arnoldi decomposition under Gram-Schmidt process for some upper Hessenberg matrix $H_m$, 
\beqq\label{eq108}
(I-\gamma A)^{-1}V_m=
C_1^{1/2} S_{1,1} C_1^{-1/2}V_m=V_m H_m.
\eeqq
The orthogonality $V_m^\top V_m=I$ gives
\[
H_m=V_m^\top C_1^{1/2} S_{1,1} C_1^{-1/2}V_m.
\]
Simplifying 
 (\ref{eq108}) yileds
\[
V_m (I-H_m^{-1})=\gamma A V_m. 
\]
Let $P_m:=V_mV_m^\top $ be  the orthogonal projection onto $Q_m(A,v)$, and $A_m$ be the restriction of $A$ on $Q_m(A,v)$, 
\[
A_m=P_m A P_m=P_m \gamma^{-1} (I-H_m^{-1}) P_m.
\]
Let $v=C^{1/2} u(0)$ and $V_m =C^{1/2} W_m$.
The construction of $W_m$ ensures its columns lying in the range of $V_C$.  Theorem~\ref{Thm3.1} indicates
\begin{eqnarray}
&&\| \{V_C f_l(S_{1,1}) V_C^\top-W_m f_l(H_m) W_m^\top C\} u(0) \|_C\\
&=&
\| C^{1/2}\{V_C f_l(S_{1,1}) V_C^\top-W_m f_l(H_m) W_m^\top C^{1/2}\} v \|
\\
&=&
\|V_C^\top  C^{1/2} V_C f_l(S_{1,1}) V_C^\top-V_C^\top V_m f_l(H_m) W_m^\top C^{1/2}\} v \|
\\
&=& \left\|C_1^{1/2} \varphi_l( \gamma^{-1}(I-S_{1,1}^{-1})) C_1^{-1/2}v - P_m\varphi_l( \gamma^{-1} (I-H_m^{-1}) ) P_m v \right \|\\
&=& \|\varphi_l(A) v-\varphi_l(A_m) v \|
\le \frac{C(l,\gamma)}{m^{l/2}} \|u(0)\|_C,\label{u0term}
\end{eqnarray}
where $C(l,\gamma)$ is a constant depending on $l,\gamma$, but independent of $m$ or $A$.  Take $l=1$ for the $u(0)$-term. Similar arguments apply to  the $u'(0)$-term. 
Lastly, for the first term involving $x(0)$,  since  $V_CS_{1,1} V_C^\top =S$,   the difference of $\varphi_1$ tends to $0$ as $m\to \infty$.
\end{proof}
  
\subsubsection{Linear convergence}
When  $G$ is  positive definite, we can derive  (\ref{error1})   under the framework  in~\cite{Hoch}.
 We  estimate the error $\{V_C f(S_{1,1})-  W_m^{(0)} f(H_m)^{(0)} {W_m^{(0)} }^\top  C V_C \} v$ in (\ref{x_bd})
 for any nonzero vector 
 $w=V_C v$ as follows.
Since $f$ in (\ref{def_f}) is an analytic function on $\IC-\{0\}$,  $f(S_{1,1}) v$ and its Krylov space approximation  have the   Cauchy integral expression (Definition 1.11~\cite{Higham}) 
 \begin{eqnarray}
&&V_C  f(S_{1,1})  v=\frac{1}{2\pi i}\int_\Gamma f(\lambda) V_C (\lambda I- V_C^\top  S V_C )^{-1}  v d\lambda=\frac{1}{2\pi i}\int_\Gamma f(\lambda)  (\lambda I- S )^{-1}  w d\lambda,\label{eq69}\\
&&  W_m f(H_m)  W_m^\top C V_C v=\frac{1}{2\pi i}\int_\Gamma f(\lambda) W_m (\lambda I-H_m)^{-1} W_m^\top w d\lambda,\label{eq70}
\end{eqnarray}
where $\Gamma$ can be a closed contour enclosing all the eigenvalues of $S_{1,1}:=V_C^\top  S V_C$, but not enclosing $0$.
  The following  shows the effectiveness of  $C$-orthogonality Arnoldi algorithms in solving $x_\cR(t)$ of (\ref{sys}) under (\ref{assGC}).
   Since $\rho_0/r<1$, the error tends to $0$ as $m\to \infty$. The proof is listed in the appendix. 
\begin{theorem}\label{thm4}
Suppose $C,G$ satisfy (\ref{assGC}). Then $ \cF_C(B)$ are bounded by    $\cD(c_1,\rho_1)$ with   a real number $c_1>\rho_1$, i.e., $0$ not inside $\cF_C(B)$ and thus    Prop.~\ref{Bound1} indicates that  $\cF_C(S)$ is bounded by a disk   $\cD(c_0, \rho_0)$ with $c_0>\rho_0$.   Take $\Gamma$ as one circle with centre $c_0$ and radius $r\in (\rho_0, c_0)$. Then  \beqq\label{error1}
\| x_\cR (t)-x_a(t)\|_C \le \max_{\lambda\in \Gamma}( |f(\lambda)|\|x(0)\|_C+|f_1(\lambda)|\|u(0)\|_C+|f_2(\lambda)|\|u'(0)\|_C ) \cdot  \frac{4 }{ (r-\rho_0)}   (\frac{\rho_0}{r})^m.
\eeqq

\end{theorem}

\subsection{ Upper bounds  $E(\gamma)$ with $h/\gamma$  fixed}
From Prop.~\ref{FCB}, $\cF_C(B)$ lies in the right half plane with $c_1>\rho_1>0$,
 \[ \cF_C(B)\subset \cD(c_1,\rho_1).\]
 Let
$\mu_1:=c_1-\rho_1, \mu_2:=c_1+\rho_1$ be  lower and upper bounds for $\Re(\cF_C(B))$, respectively.  Since 
M{\"o}bius transformations map generalized circles to generalized circles, the function $g$ maps
  $\cD(c_1,\rho_1)$ in the $\mu$-plane  to $\cD(c_0, \rho)$ in the $\lambda$-plane, where $c_0, \rho$ are  functions of $\gamma$, 
 \beqq\label{rho_fun}
 c_0=\frac{1}{2}\left((1+\gamma/\mu_2)^{-1}+(1+\gamma/\mu_1)^{-1}\right),\;
 \rho=\frac{1}{2}\left((1+\gamma/\mu_2)^{-1}-(1+\gamma/\mu_1)^{-1}\right).\eeqq
 Consider the $\varphi_0$ case,  $f$ defined in (\ref{def_f}) with $t=h$, 
 \[
 f(\lambda)=\exp(-(h/\gamma) (\lambda^{-1}-1)).
 \]
One upper bound for the right hand side of  (\ref{error1}) is given by
\beqq\label{eq82}
|f(c_0+r)|\cdot \frac{4}{(r-\rho)}\cdot (\frac{\rho}{r})^m.
\eeqq
To simplify the computation, 
choose $\Gamma$ to be one  circle tangent to the imaginary axis at $0$, sharing the same centre with $\cD(c_0, \rho)$, i.e., $r=c_0$ is chosen. Here we are interested in asymptotic results, i.e.,  $m\to \infty$, thus
for the sake of simplicity, 
we
omit the absolute constant $4$ in (\ref{eq82}), 
\beqq\label{Efun}
E(\gamma):=\exp((h/\gamma)(1-(2c_0)^{-1}))\left(\frac{\rho}{c_0}\right)^m \frac{1}{(c_0-\rho)}.
\eeqq

\subsubsection{$\phi_0$ functions }\label{case1}
Suppose  the eigenvalue information on $B_{1,1}$ is not available. It is natural to choose $\gamma$  proportional to  $h$, as in  \cite{AWESOME}.
The following  computation gives qualitative analysis on $E$ with respect to  $\gamma$. Here we focus on the  $\varphi_0$ case. Arguments can be applied to other $\varphi_k$ functions after some proper modifications.  The proofs are tedious, and placed in the appendix. 
 Introduce $\rho_*, \gamma_*$ as follows, where $c_0(\gamma_*)=1/2$: 
 \[ \gamma_*=\sqrt{\mu_1\mu_2},\; \rho(\gamma_*)=\rho_*:=\frac{1}{2}\frac{\sqrt{\mu_2}-\sqrt{\mu_1}}{\sqrt{\mu_2}+\sqrt{\mu_1}}. \]

The following shows  that   the base $\rho/c_0$ of $(\rho/c_0)^m$ in $E$ gets smaller, as $\gamma$ gets close to $0$. In particular, at $\gamma=\gamma_*$, 
\[
\frac{\rho}{c_0}=\frac{\sqrt{\mu_2}-\sqrt{\mu_1}}{\sqrt{\mu_2}+\sqrt{\mu_1}}.
\]

\begin{prop}\label{num}
As $\gamma$ increases in $[0,\infty)$, the radius  ratio 
\[
\frac{\rho}{c_0}=\frac{(\mu_2-\mu_1)\gamma}{\mu_1(\mu_2+\gamma)+\mu_2(\mu_1+\gamma)}
\] 
 increases. 
\end{prop}

 Prop.~\ref{prop3.5} indicates  that when $\delta=h/\gamma$ is kept fixed,   the slope of $E(\gamma)$  decreases as $\gamma$ increases from $0$ to $\infty$. The graph of  $E(\gamma)$ asymptotically  looks like a $\cap$-shaped curve. 
In particular, $E(\gamma)$ can decay rapidly when 
  $\gamma$ is sufficiently larger than $\mu_2$.
  

%
%
%

\begin{prop}\label{prop3.5}
  Let  $\delta=h/\gamma$ fixed. 
Let $\omega=\mu_1/\mu_2$ and \[
\epsilon(\gamma)=\delta-\frac{2m\omega}{1+3\omega}(1+\sqrt{\omega})^2-\frac{(1+\sqrt{\omega})^2}{1+\omega}.
\]
For  $\omega:=\mu_1/\mu_2$ close to $0$ with $\epsilon>0$ , we have
\[
- \frac{d}{d\gamma}\log E(\gamma)\ge  (\sqrt{\mu_1}+\sqrt{\mu_2})^{-2} \epsilon.
\]
Then $E(\gamma)$ has the exponential decay for $\gamma>\mu_2$,
\[
E(\gamma)=E(\mu_2)\exp(-\epsilon (\gamma-\mu_2)(\sqrt{\mu_1}+\sqrt{\mu_2})^{-2}).
\]

\end{prop}
\begin{rem} With  similar calculus computation, 
the error upper bound function $E(\gamma)$ behaves as one ``flat" function  for  small $\gamma$. In particular, 
when $\gamma\le \gamma_*$, (\ref{eq102}) gives
\beqq\label{eq103}
\frac{m}{\gamma}  (\frac{2\mu_1\mu_2}{(\mu_1+\mu_2) \gamma+2\mu_1\mu_2})\xi
=2\mu_1\mu_2\frac{m}{\gamma}  \cdot \frac{ (
(\mu_1+\mu_2)\gamma+2\mu_1\mu_2
)}{
\mu_1(\mu_2+\gamma)^2+
\mu_2(\mu_1+\gamma)^2
}
\ge \frac{m}{\gamma}  \left(\frac{(\mu_1+\mu_2)\gamma+2\mu_1\mu_2}{(\sqrt{\mu_1}+\sqrt{\mu_2})^2}\right).
\eeqq
Hence, with $\gamma$ sufficiently close to $0$, the lower bound in (\ref{eq103}) will  exceed $\delta$ eventually, which 
indicates  the increase of $E(\gamma)$,
$ \frac{d}{d\gamma}\log E(\gamma)>0$ in (\ref{eq102}). 
However, as $\mu_1$ is very close to $0$, the increases of $\log E$  could be very slow,  $O(\log \gamma)$. For instance,  at    $\gamma=\mu_1$, 
\[
\frac{d}{d\log \gamma}\log E(\gamma)=2\mu_1\mu_2 m\frac{\mu_1+3\mu_2}{(\mu_1+\mu_2)^2+4\mu_1\mu_2}\le 
6\mu_1m.\]
\end{rem}

\subsubsection{Higher order functions $\varphi_k$ }

The  phi-functions 
\[
\varphi_0(z)=\exp(z),\; \varphi_k(z)=z^{-1} (\varphi_{k-1}(z)-1/((k-1)!))
\]
are initially  proposed to serve as error bounds for the matrix exponential function,  e.g., Theorem 5.1 in \cite{Saad92}. In applications, one can use any function $\varphi_k$, $k>0$ to compute  $\exp(-B_{1,1}^{-1} h) v$.
Researchers~\cite{AWESOME} observe dissimilar error behaviours,   even though 
   two equivalent phi functions  are computed based on
   Krylov subspace approximations, 
\begin{eqnarray}
&&\varphi_0(-h B_{1,1}^{-1})B_{1,1}v,\\
&&-h\varphi_1(-h B_{1,1}^{-1})v+B_{1,1}v,\label{eq122}
\end{eqnarray}  
 
 Here we focus on the computation framework in (\ref{eq122}).
With small Krylov dimensions, the error   mainly originates from the Krylov approximation error of $h\varphi_1(-h B_{1,1}^{-1})$.
To estimate the error, we can  choose $f$ in (\ref{error1}) to be
\beqq\label{f1_eq}
f(\lambda):=h\varphi_1(
(h/\gamma) (1-\lambda^{-1}))
=h\{(h/\gamma) (1-\lambda^{-1})\}^{-1}  \{\exp((h/\gamma) (1-\lambda^{-1}))-1\}.
\eeqq
For general $k\ge 1$,   choose  \beqq\label{phik_f}
 f(\lambda)= f(g(\mu))=h^k \varphi_k((h/\gamma) (1-\lambda^{-1}))
 =h^k \varphi_k(- h/\mu)
 \eeqq
  in estimating the error of the $\varphi_k$ case,
  \beqq\label{phi2}
\exp(-h B_{1,1}^{-1})u=u+\sum_{j=1}^{k-1}(-h B_{1,1}^{-1})^j u+ (-h)^k \varphi_k(-h B_{1,1}^{-1})\cdot ( B_{1,1}^{-1})^k  u.
\eeqq
 Prop.~\ref{h_bd} shows that $1/k!$ is one  upper bound for each $\varphi_k$ for $k\ge 1$ and thus $f$ has an upper bound $h^k/k!$.
This new upper bound  mainly brings two adjustments to the original $\cap$-shaped  error bound. First, the exponential fast dropping  under  large $h$ disappears, 
since the upper bound for this function $f$ is lifted to an increasing function $h^k/k!$.  Second,   polynomial decaying  under  small $\gamma$  can be obtained, in contrast to the original stagnation in   the $\varphi_0$-case.
  \begin{prop} \label{h_bd}
  Consider integers $k>0$. Let $f$ be given in (\ref{phik_f}).
  Then   with $g(\mu)=(1+\mu^{-1}\gamma)^{-1}$,   $|f(g(\mu))|$ can be bounded by $h^k/(k!)$.
    \end{prop}
  \begin{proof}  
 Let    $\lambda=g(\mu)$. Claim: for each positive integer $k$, we have
 \[
 |\varphi_k (-h\mu^{-1})|\le {(k!)}^{-1}.
 \]
 By Taylor's expansion Theorem.  
 if $z<0$, then with $\xi $ between $0$ and $z$,
 \[
 \varphi_k(z)=z^{-1} \left(\exp(z)-1-\sum_{j=1}^{k-1}\frac{z^j}{j!}\right )= \frac{\exp(\xi) z^k/k!}{z^k}=\frac{\exp(\xi)}{k!}.
 \]
 Since  $\xi\in [-h\mu^{-1},0 ]$, then   \beqq\label{phi_1}  |f(g(\mu))|=
|  h^k\varphi_k(h\mu^{-1})|\le( k!)^{-1} \max_{\xi} |\exp(\xi)|={(k!)}^{-1} h^k.
  \eeqq

  \end{proof}

\begin{prop}
Consider $h$ in proportional to $\gamma$,  $\delta=h/\gamma$. 
Error bounds  corresponding to  the $\varphi_k$ case can be described  by
\beqq\label{E_phi_k}
 E(\gamma):=h^k ( \frac{\rho}{c_0})^m\frac{1}{c_0-\rho}. \eeqq
Then 
\[ \frac{d\log E}{d\log \gamma}\ge k+1,\; \forall \gamma>0.\] 
\end{prop}
\begin{proof} From (\ref{E_phi_k}), we have
\[
\log E=k\log (\delta \gamma)+m\log \frac{\rho}{c_0}-\log(c_0-\rho). \]
To explore the dependence on $\gamma$, taking derivative with respect to $\gamma$ yields
\begin{eqnarray}
&& \frac{d}{d\gamma}\log E(\gamma)=\frac{d}{d\gamma} \{ k\log \gamma+m\log \frac{\rho}{c_0}-\log (c-\rho)\}\\
&=&\frac{k}{\gamma}
+2m \left((\frac{1}{\mu_1}+\frac{1}{\mu_2}) \gamma^2+2\gamma\right)^{-1}+(\mu_1+\gamma)^{-1}.
\label{eq97}
\end{eqnarray}
Hence, for all $\gamma>0$, we have
\[
\frac{d\log E}{d\log \gamma}=k+2m(
(\frac{1}{\mu_1}+\frac{1}{\mu_2}) \gamma+2
)^{-1} +1-\frac{\mu_1}{\mu_1+\gamma}>k+1.
\]
\end{proof}

\section{Simulations}\label{simulations}

\label{sec_problem}
Previous work in~\cite{chen2018transient} and~\cite{AWESOME} is recalled to illustrate the stability issue in solving semi-explicit DAEs by the ordinary Arnoldi method. 

\subsection{Stability Problems of DAEs}
\label{sec:sing}
We start from a one tank lumped RLC model as shown in Fig.~\ref{Fig_RLC_1tank}. A step input current source $I_S$ with rise time TR$=1ps$ is applied. The DAEs $C\dot{ x}+Gx=u$
of the one tank RLC follow the semi-explicit structure as expressed in Eq.~(\ref{Eq_RLC_1tank}). The node voltages and branch currents in the state vector are marked in Fig.~\ref{Fig_RLC_1tank}.

\begin{table}[hbt]
\begin{eqnarray}
\label{Eq_RLC_1tank}
\begin{pmatrix}
0 & & & \\
 & 0 & & \\
 & & C1 & \\
 & & & L1 
\end{pmatrix} 
\begin{pmatrix}
\dot{v_1} \\ \dot{v_2} \\ \dot{v_3} \\ \dot{i_L} 
\end{pmatrix}
+	
 \begin{pmatrix}
\frac{1}{R_1} + \frac{1}{R_2} & -\frac{1}{R_1} & & \\
 -\frac{1}{R_1} & \frac{1}{R_1} & & 1\\
 & & 0 & -1\\
 & -1 & 1 & 0
\end{pmatrix} \begin{pmatrix}
v_1 \\ v_2 \\ v_3 \\ i_L 
\end{pmatrix} = 
\begin{pmatrix}
 I_{bias}\\
0 \\
-I_S \\
0
\end{pmatrix} 
\end{eqnarray}
\caption{}
\vspace{-0.5cm}
\end{table}

First, rational Krylov subspace is constructed through Arnoldi iterations in ~(\ref{Ord_Arnoldi}) in the simulation to compute the matrix exponential with lower order $\varphi_0$ functions, i.e., ~(\ref{solap}). We set $h=TR$ for the input transition and use fixed step size for the stable stage. Since $C$ is singular, we do observe the failure of the application of the ordinary Arnoldi algorithm.  Fig.~\ref{Fig_RLC_1tank_res} depicts the node voltages and the solution residual in the simulation, showing that the residual terms on algebraic variables $v_1$ and $v_2$ start to increase at early stage and generally drive the whole system to an incorrect converging direction. 
Trapezoidal  method results with fixed step size $100ps$ are plotted as comparison which show a deviation from exact solution as well. 
\begin{figure}[t]
	\centering
	\includegraphics[trim = 0cm 0cm 0cm 0cm, clip=true, keepaspectratio, width = 0.5\textwidth
	]
	{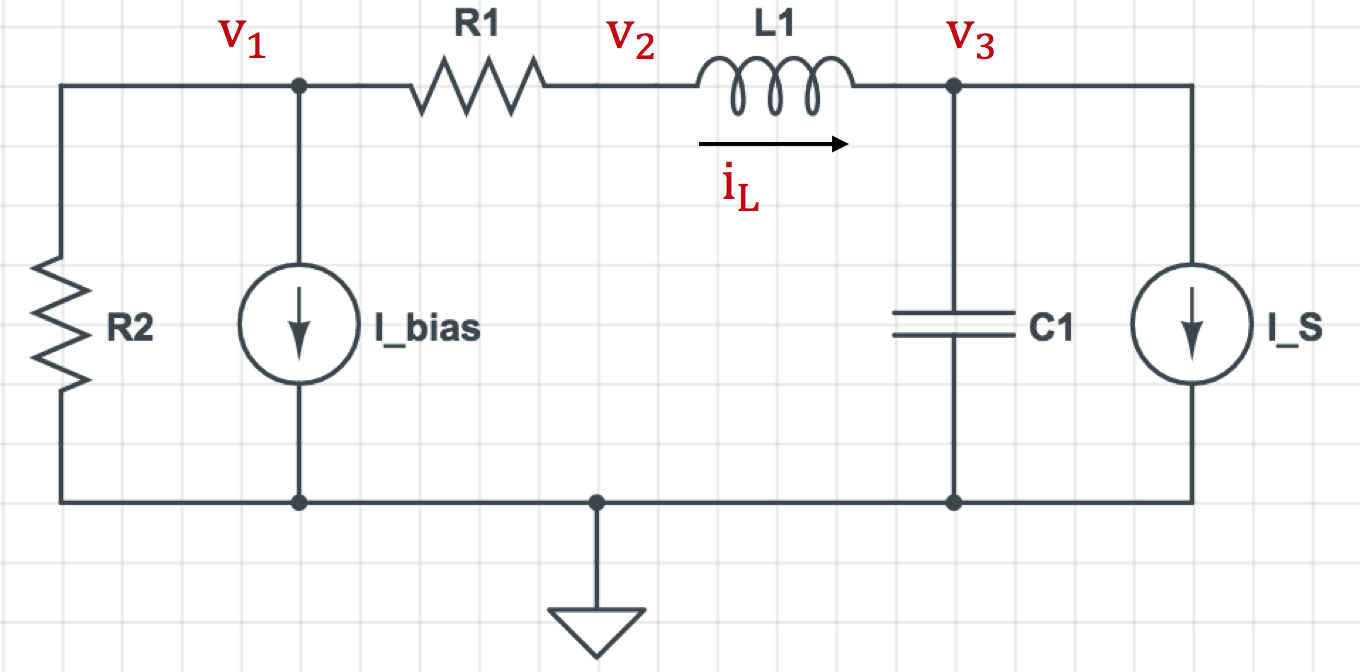}
	\caption{ 
{\small{One tank RLC with $R1 = 100\mu\Omega, L1 = 0.5nH, C1 = 0.5nF$ and $R2 << R1$.}}
	}
	\label{Fig_RLC_1tank}
\end{figure}

\begin{figure}[t]
	\centering
	\includegraphics[trim = 0.7cm 0.1cm 0.4cm 0.7cm, keepaspectratio, width = 0.35\textwidth
	]
	{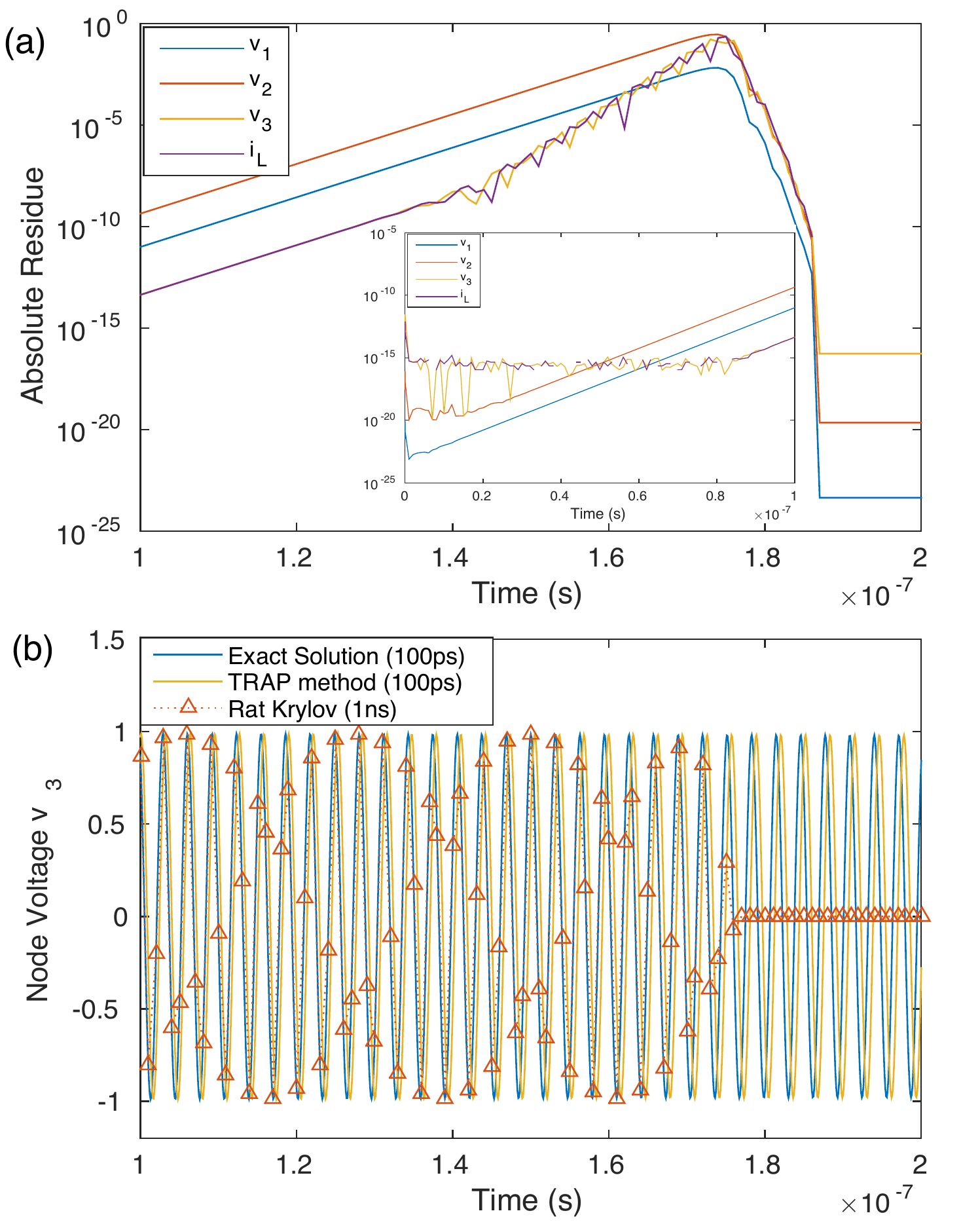}
	\caption{ 
{\small{Simulation results of the one tank RLC (Fig.~\ref{Fig_RLC_1tank}). (a) absolute value of $residual = C\dot{x}(t) + Gx(t) - u(t)$ for each variable in $x(t)$; (b) simulation results on $v_3$ with rational Krylov subspace method as well as Trapezoidal method, exact solution is included as comparison.}}
	}
	\label{Fig_RLC_1tank_res}
	\vspace{-0.5cm}
\end{figure}
\label{sec:sens}

From the observations on ill-conditioned system from DAEs, the numerical error occurs in the calculation of algebraic variables and could result in stability issues in later simulation stage. To eliminate the error in the nullspace $\mathcal{N}({G}^{-1}{C}) = \mathcal{N}({C})$, the algebraic variables are set to zero in the Arnoldi process. The technique was called implicit regularization~\cite{chen2018transient}. 
\begin{eqnarray}
v= \begin{pmatrix} v_R \\ v_N \end{pmatrix} \Rightarrow P_C v = \begin{pmatrix} I  & 0\\  0 &0 \end{pmatrix} \begin{pmatrix} v_R \\ v_N \end{pmatrix} = \begin{pmatrix} v_R \\ 0 \end{pmatrix}.
\end{eqnarray}

\noindent Since $C$ is diagonal,  the matrix  { $P_C$} only contains an identity matrix for the differential variables and zeros for the algebraic variables. The approach forces the computations in the range of $C$. 

Simulation results of one tank RLC with implicit regularization are shown in Fig. \ref{Fig_Results_RLC_1tank}, which fit the exact solution. Residuals of $v_3$ and $i_L$ remain at a low level ($\approx 10^{-15}$) when the input current is stable. The other variable could be solved  algebraically and the system no longer suffers from the singularity problem. More discussions on stability can be found in (\cite{AWESOME}).

\begin{figure}[t]
	\centering
	\includegraphics[trim = 0.7cm 0.1cm 0.4cm 0.7cm, keepaspectratio, width = 0.35\textwidth
	]
	{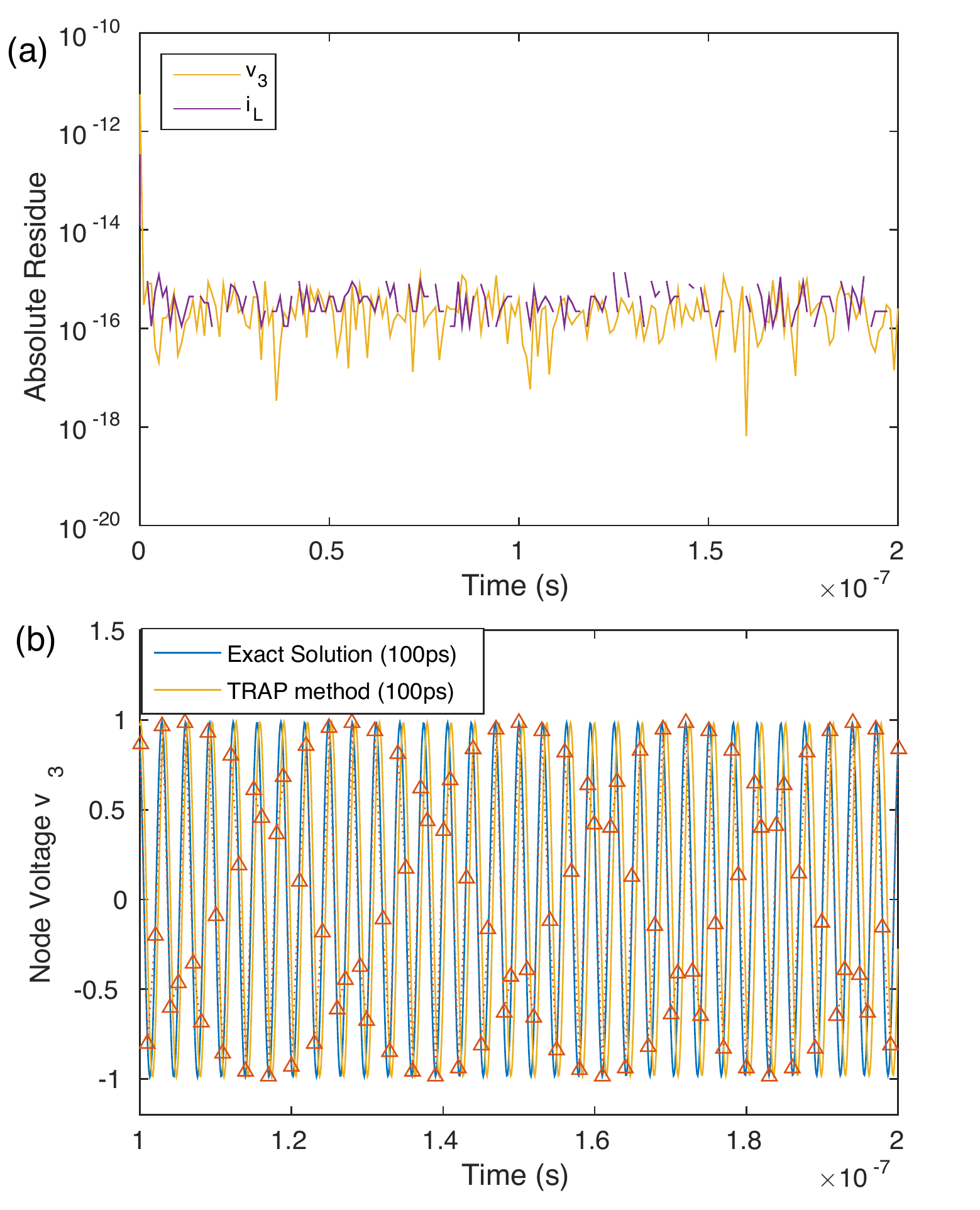}
	\caption{ { \small{Simulation results of the one tank RLC (Fig.~\ref{Fig_RLC_1tank}) with implicit regularization. (a) The absolute residual no longer increase and (b) simulation results well fit the exact solution. Node voltages $v_1$, $v_2$ are solved algebraically.}}
	}
	\label{Fig_Results_RLC_1tank}
	\vspace{-0.6cm}
\end{figure}

This simple example illustrates  whether  the numerical range of $B$ is located in the right half plane or not
affects the sensitivity of numerical integration methods.  Indeed, 
 since the matrix $P_C G^{-1} C$ is 
\[
\left(
\begin{array}{cccc}
0  & 0  &0 & 0   \\
0  & 0  &0 &  0 \\
0  & 0  & 5\times 10^{-14} & 5\times 10^{-10}   \\
0  & 0  & -5\times 10^{-10} &0
\end{array}
\right),
\]
 $\cF_C(B)$ is the ellipse  with centre $(2.5\times 10^{-14},0)$ and semi-major axis $5\times 10^{-14}$ and semi-minor axis $5\times 10^{-10}$.
By (\ref{HS}) and Prop.~\ref{Bound1},
the Rayleigh quotient of the matrix $H_m$ always lie in the image of the ellipse under the function $g$. Thus, $\cF_C(H_m)$ lies in the disk $\cD(1/2,1/2)$. 
In contrast, the ordinary Arnoldi iterations generate upper Hessenberg matrix $H_m$, whose numerical range $\cF(H_m)$ 
does not necessarily  lies in $\cD(1/2,1/2)$, since part  of  $\cF(B)$ even lies in the left half plane. 

 \begin{figure}
 \begin{center}
 \includegraphics[width=0.35\textwidth]{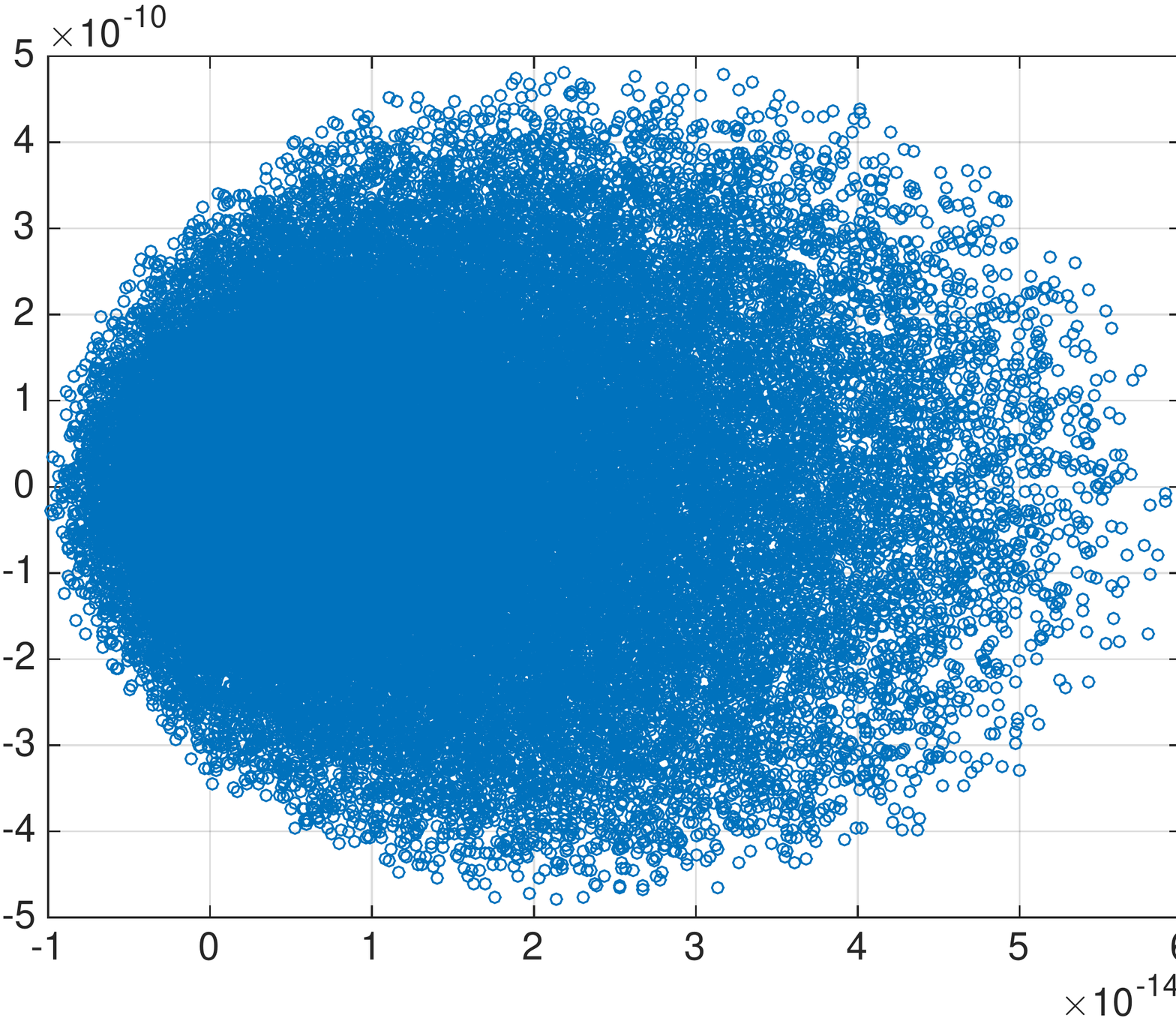} 
 \includegraphics[width=0.35\textwidth]{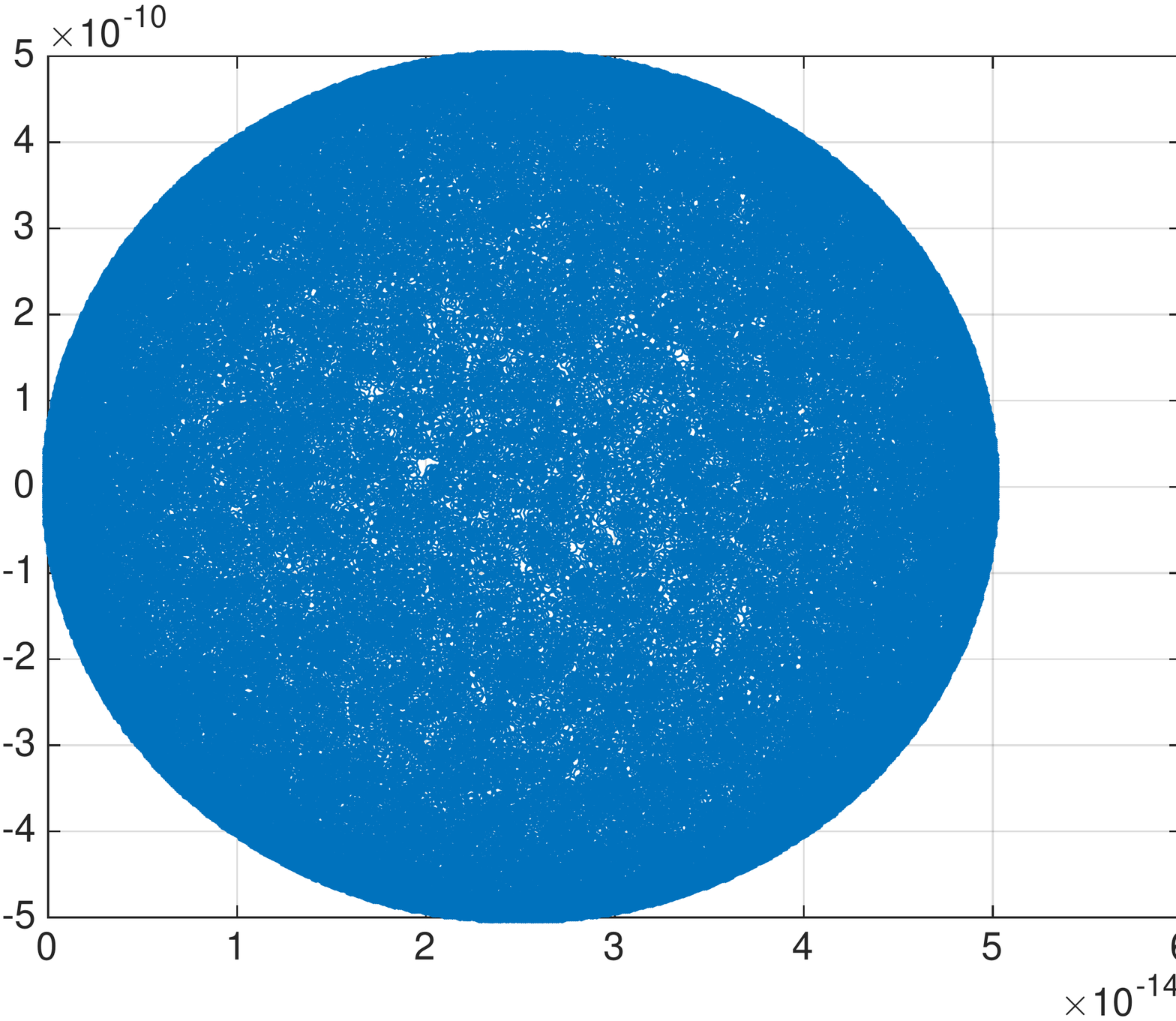} 

   \caption{  Illustration of $\cF(G^{-1}C)$(left) and $\cF_C(G^{-1}C)$(right) under $5\times 10^5$ Rayleigh quotient realizations from $\IC^4$. }
     \label{FB} \end{center}
\end{figure}

\subsection{RLC networks}
To illustrate the performance of the proposed Arnoldi algorithm on the case  with $G$ only positive semi-definite, we use  one PDN, consisting of $260$ resistors, $160$ capacitors and $160$ inductors. 
 The system matrix $C$ is positive semi-define and symmetric ( actually diagonal). The matrix $G$  is positive semi-definite, but not symmetric. The eigenvalues of $ B_{1,1}=V_C^\top  G^{-1} CV_C$ are in the range of $[10^{-17}, 10^{-8}]$. The distribution of the eigenvalues is plotted in Fig.~\ref{Fig_RLCmesh_Eigens}. 
 The transient response of the RLC mesh circuit is calculated with a single step integration. 
 Assume the slope of input current source is unchanged within the current step. 
 Starting from zero initial state $x(0)$, the response $x(h)$ of circuit at time $h$ is derived.
 The exact solution is computed by directly solving differential equations and algebraic equations in (\ref{eq1},\ref{eq2}).

The shift parameter $\gamma$ is set as $h/2$ empirically.
 The matrix exponentials  in the solution  are evaluated at different time step sizes $h$ with increasing dimension $m$ of Krylov subspace. For simplicity, we consider $x(0)=0=u(0)$ and the solution is given by
 $x(h)=h^2V_C \varphi_2(-h B_{1,1}^{-1})  C_1^{-1} V_C^\top u'(0)$. 
 Since 
\[ \varphi_0(t)v=t^2\varphi_2(t)v+v+tv=t\varphi_1(t)v+v,\] then the matrix exponential $\varphi_2(-h B_{1,1}^{-1}) v$  appeared in  the solution  can be computed with a Krylov subspace approximation of either  $\varphi_0$, $\varphi_1$ or $\varphi_2$ functions. 
Consider the following three approaches to compute
the Krylov subspace approximation.
\begin{enumerate}
\item[(a)] The original Arnoldi method with implicit regularization.
\item[(b)] The original Arnoldi method with implicit regularization + numerical pruning of spurious eigenvalues.
\item[(c)]  The Arnoldi method with structured orthognality + numerical pruning of spurious eigenvalues.
\end{enumerate}
Left column to right column in Fig.~\ref{Fig_RC_error_orig} includes the distribution of absolute error after  applying approach (a), (b) and (c), respectively. 
Here the  absolute errors are focused on matrix exponentials, thus subfigures from the top row to the bottom top  shows the  absolute errors of the following
   matrix exponentials.
\begin{enumerate}
\item[(i)] $\varphi_0$ function: $ V_C^\top \varphi_0(h B_{1,1}^{-1}) V_C G^{-1} V_C^\top C_1 V_C G^{-1} u'(0)$.
\item[(ii)] $\varphi_1$ function: $h V_C^\top \varphi_1(h B_{1,1}^{-1}) V_C G^{-1} u'(0)$.
\item[(iii)] $\varphi_2$ function: $h^2 V_C^\top \varphi_1(h B_{1,1}^{-1}) C_1^{-1} V_C G^{-1} u'(0)$.
\end{enumerate}

  Experiments in Fig.~\ref{Fig_RC_error_orig}
 show that 
  the upper Hessenberg matrix can consist of many spurious eigenvalues. From (\ref{HS}) and (\ref{rho_fun}), $\cF_C(S)\subseteq \cD(1/2,1/2)$ and thus  $\cF(H_m)\subseteq \cD(1/2,1/2)$.
  The region with spurious eigenvalues is plotted in red color. When  the original Arnoldi iterations are used, the upper Hessenberg matrix  could lose the positive definite property and
the absolute  error could grow extremely high. Clearly, the issue is resolved with (iii) see the right column. Notice that for $\gamma$  close to 0,  the set $\cF(H_m)$ is  very close to $1$ from(\ref{rho_fun}), and  
  rounding errors could easily contaminate the computations of $H_m$, such that $\cF(H_m)$  fails to lie in $\cD(1/2,1/2)$.   Hence, proper numerical pruning is required.  Observe that the error reduces quickly with all $\varphi$ functions by increasing the dimension of rational Krylov subspace, which is consistent with the theorem~\ref{Thm3}. When $h$ is larger than $\mu_2$(the upper bound for real components of eigenvalues of $B_{1,1}$),  the calculation with the $\phi_0$ function   gives the best  accuracy. On the other hand, if $h$ is smaller than the spectrum, the errors (in the log-scale)  with $\varphi_1$ and $\varphi_2$  exhibit a decrease proportional to $\gamma$ in the log scale, which alleviates the error stagnation in the solution with the  $\varphi_0$ function.

\begin{figure}[hbt]
	\centering
	\includegraphics[trim = 0.1cm 0.1cm 0.1cm 0.3cm, clip=true, keepaspectratio, width = 0.38\textwidth
	]{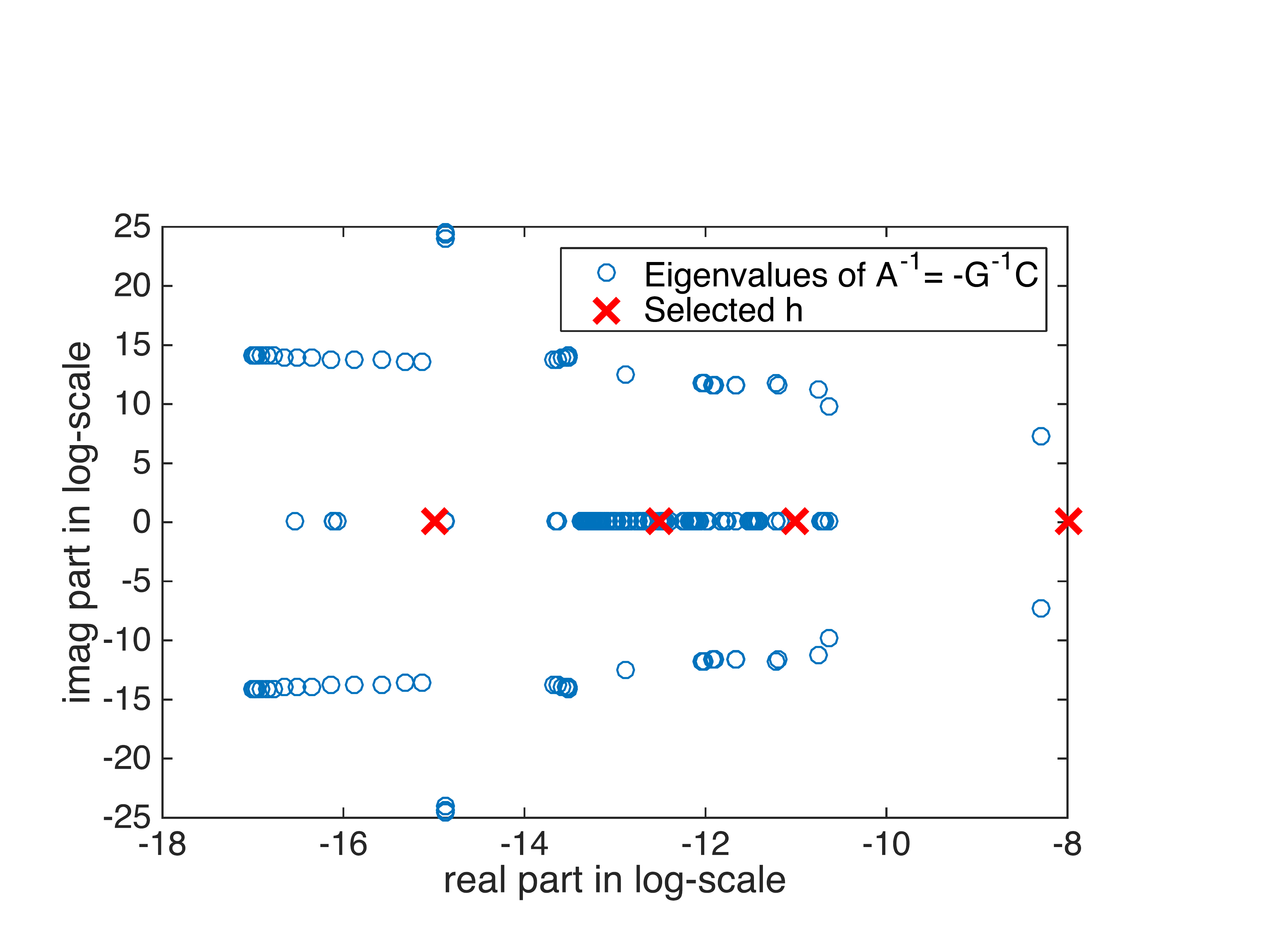}
	\caption{RLC network: eigenvalues of $B = G^{-1}C$ in log-scale.}
	\label{Fig_RLCmesh_Eigens}
	\vspace{-0.2cm}
\end{figure}

\begin{figure}[hbt]
	\centering
	\includegraphics[trim = 0.1cm 0.1cm 0.1cm 0.4cm, clip=true, keepaspectratio, width = 0.3\textwidth
	]
	{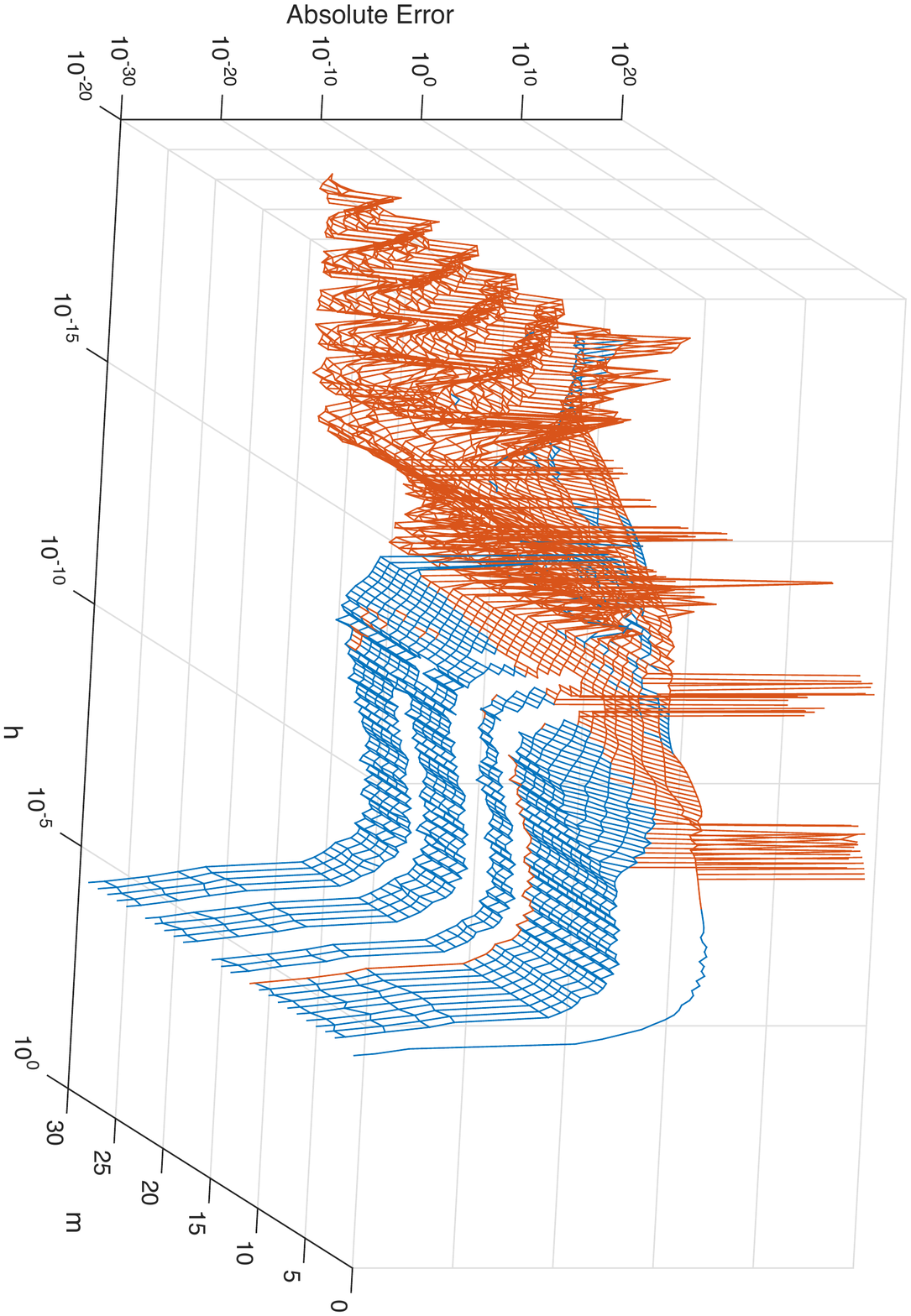}
		\includegraphics[trim = 0.1cm 0.1cm 0.1cm 0.4cm, clip=true, keepaspectratio, width = 0.3\textwidth
	]
	{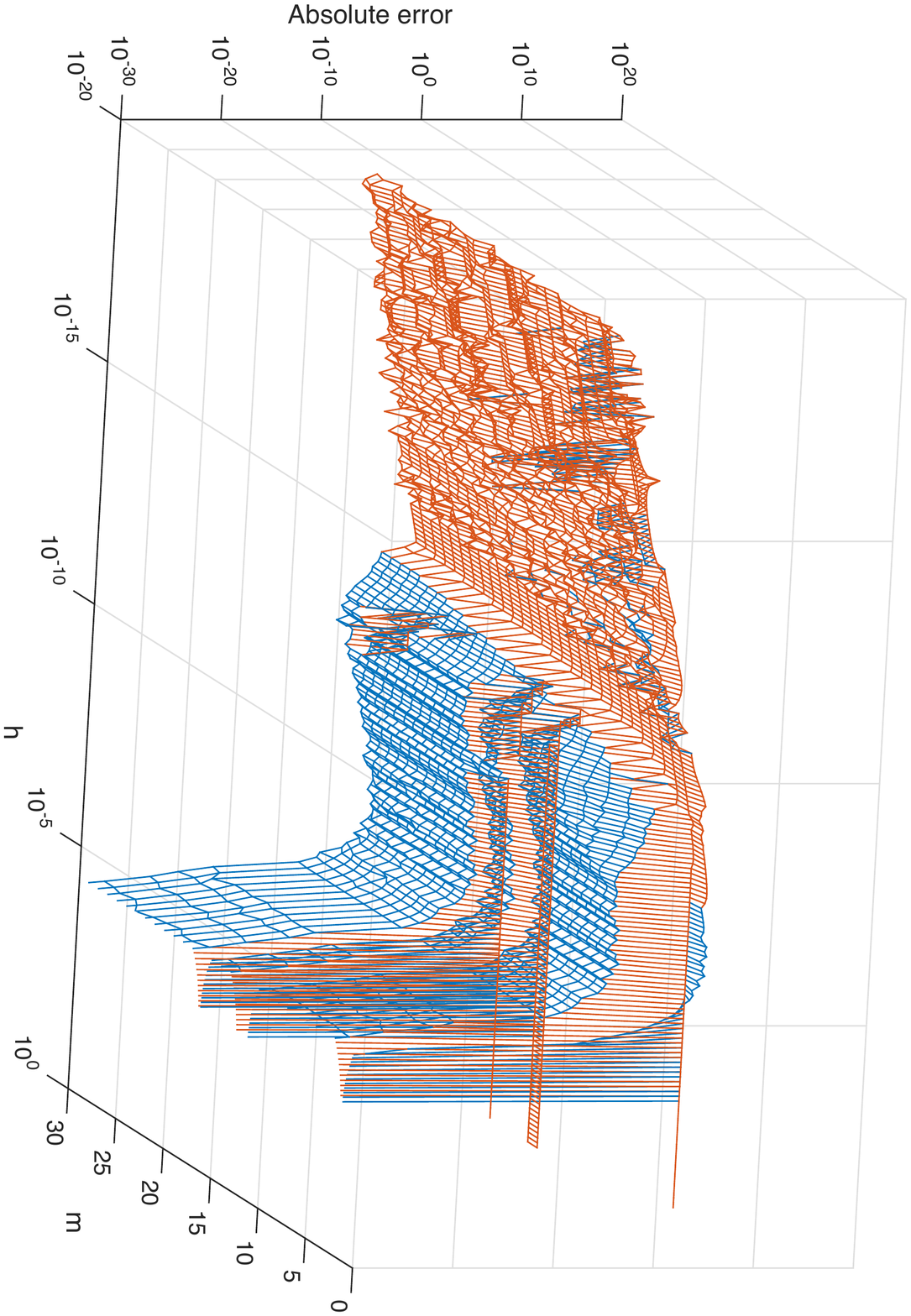}
	\includegraphics[trim = 0.1cm 0.1cm 0.1cm 0.4cm, clip=true, keepaspectratio, width = 0.3\textwidth
	]
        {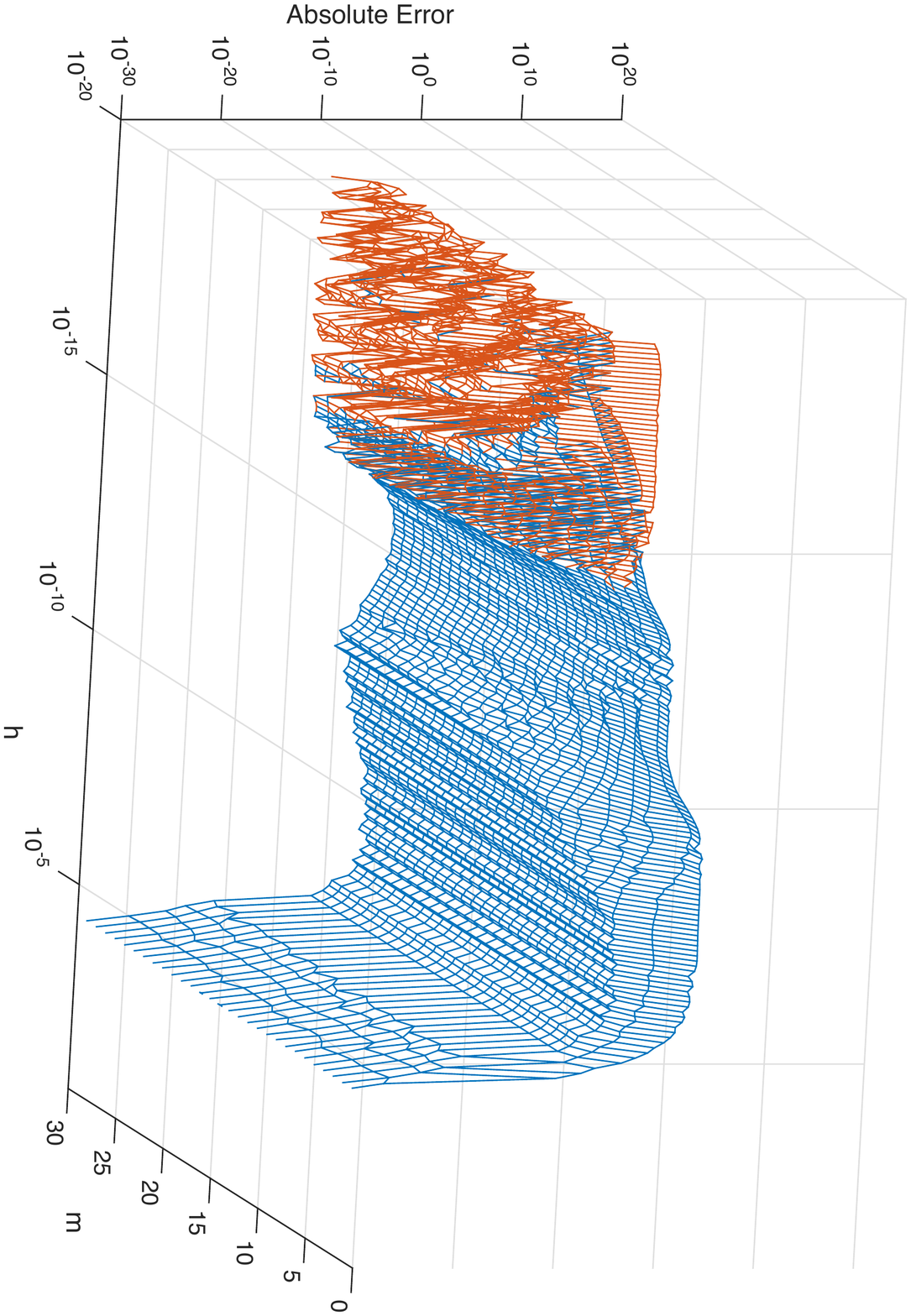}\\	
	\includegraphics[trim = 0.1cm 0.1cm 0.1cm 0.4cm, clip=true, keepaspectratio, width = 0.3\textwidth
	]
		{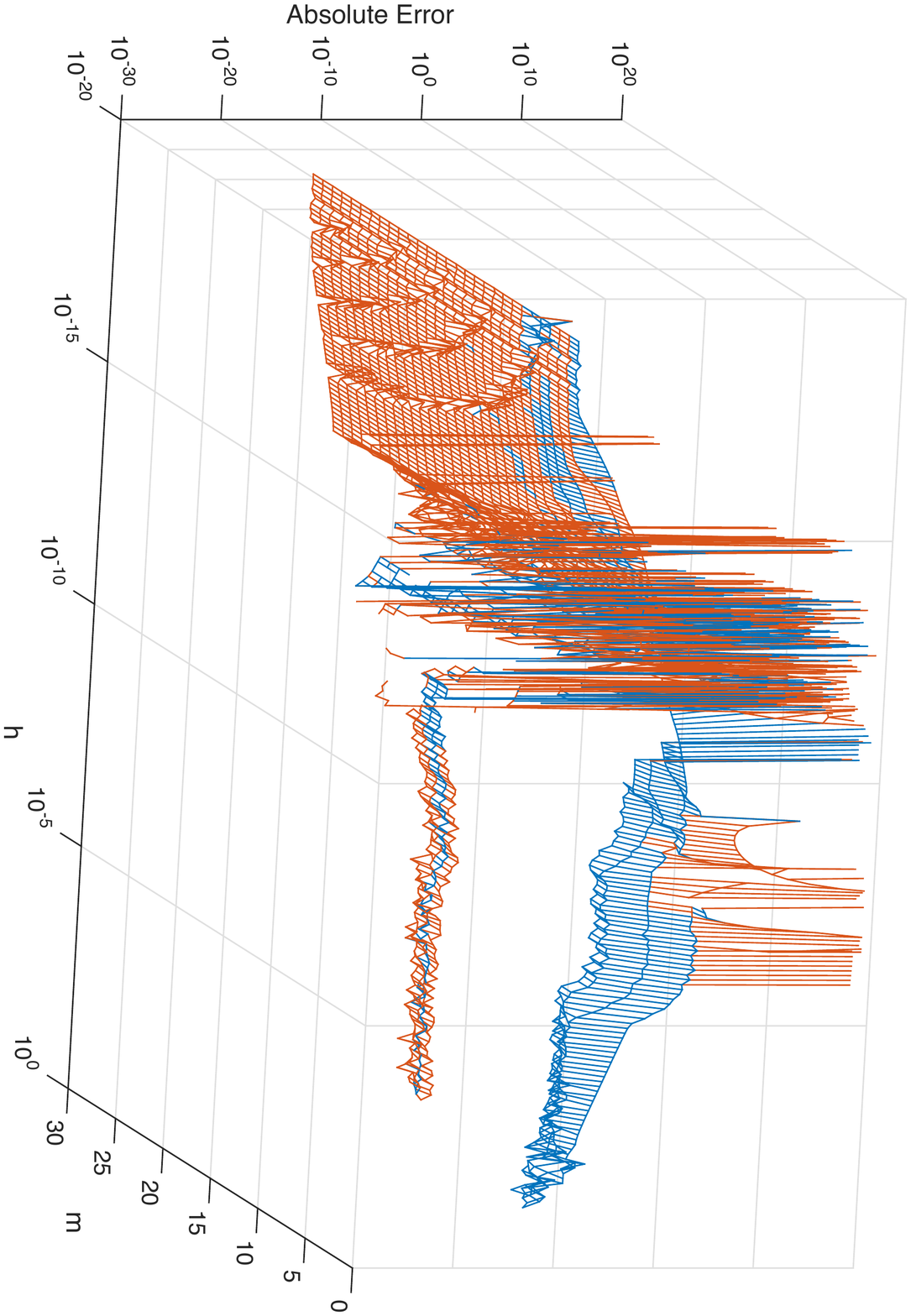}
		\includegraphics[trim = 0.1cm 0.1cm 0.1cm 0.4cm, clip=true, keepaspectratio, width = 0.3\textwidth
	]
	{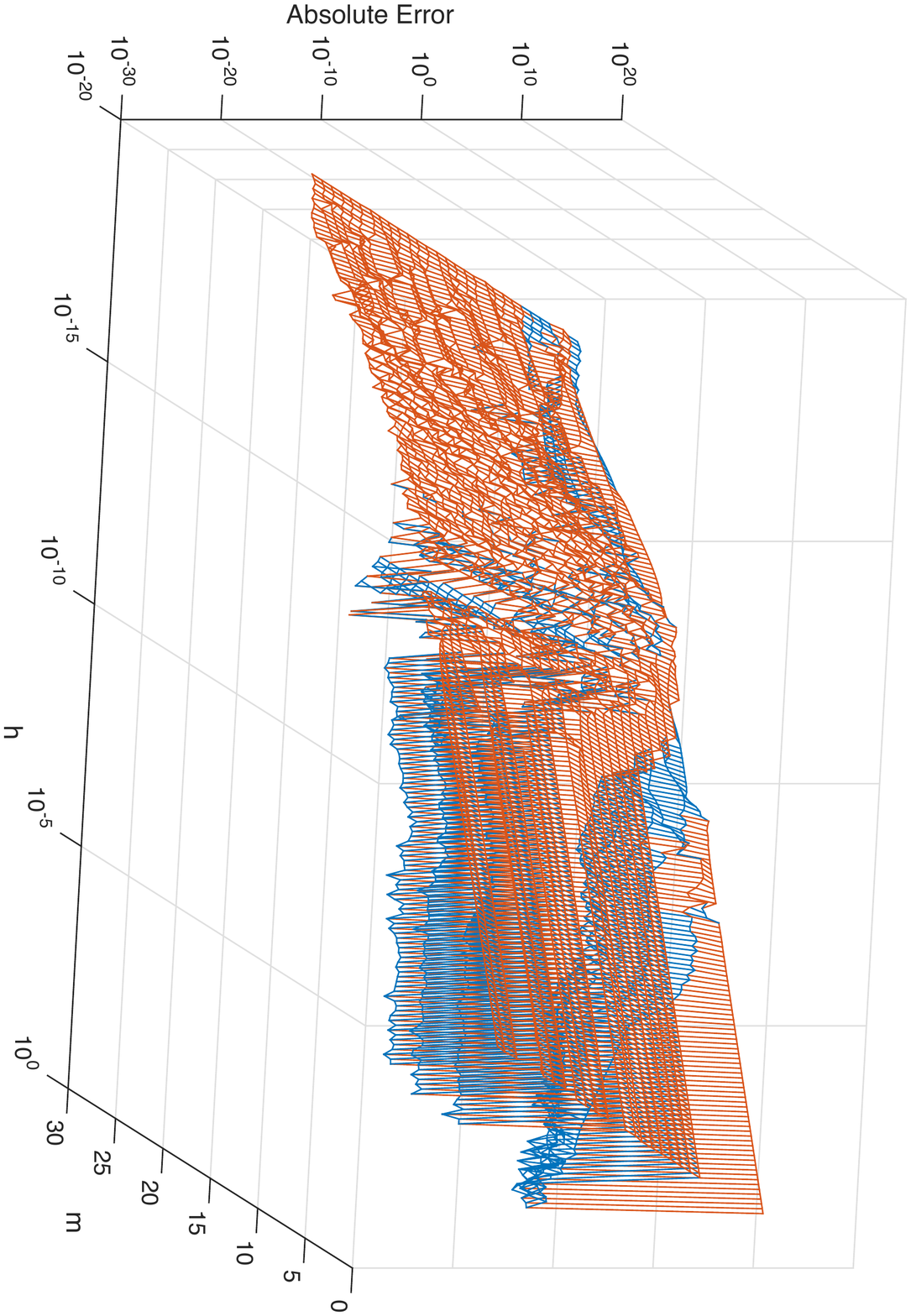}
		\includegraphics[trim = 0.1cm 0.1cm 0.1cm 0.4cm, clip=true, keepaspectratio, width = 0.3\textwidth
	]
	{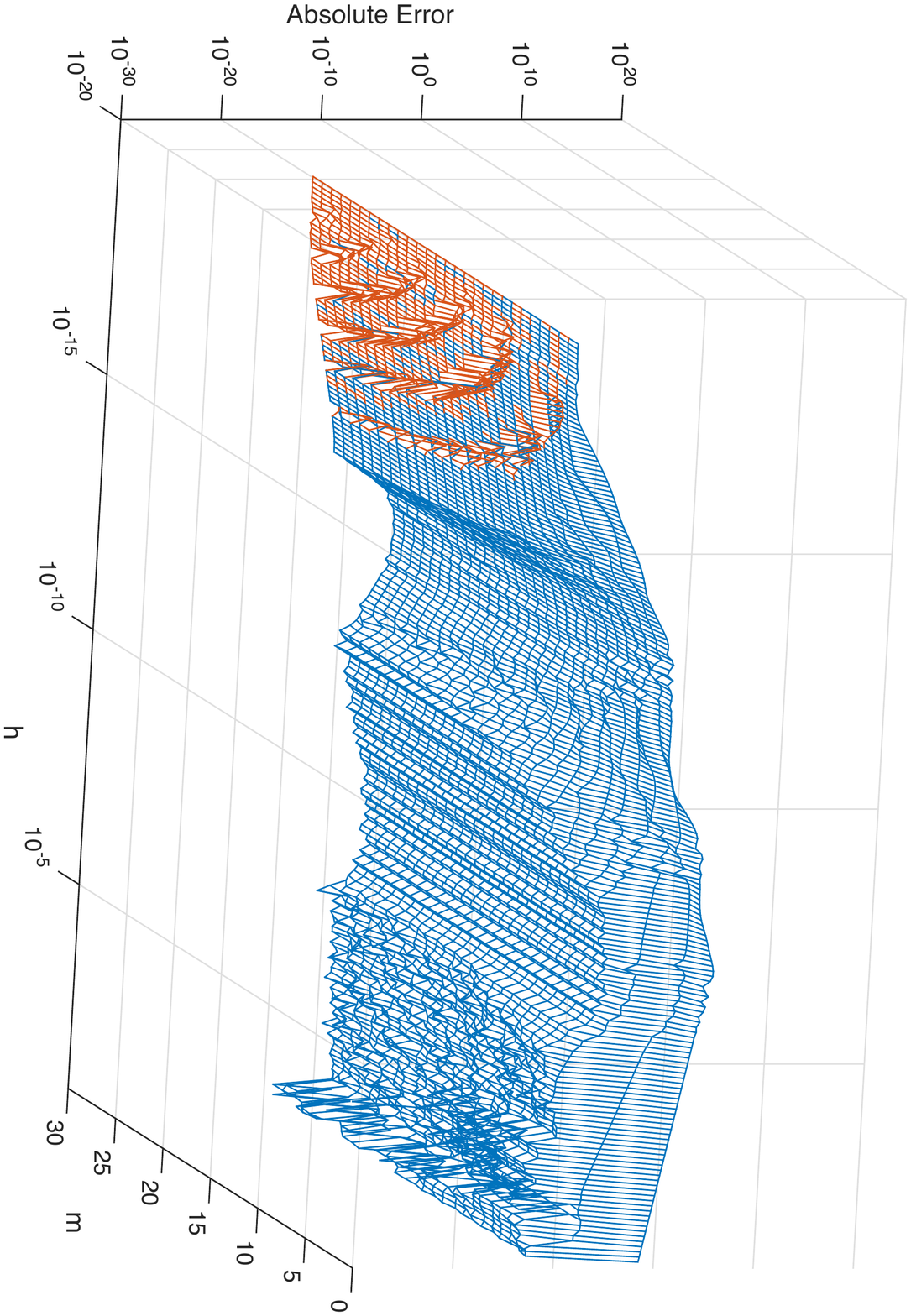}
	\\
	\includegraphics[trim = 0.1cm 0.1cm 0.1cm 0.4cm, clip=true, keepaspectratio, width = 0.3\textwidth
	]
	{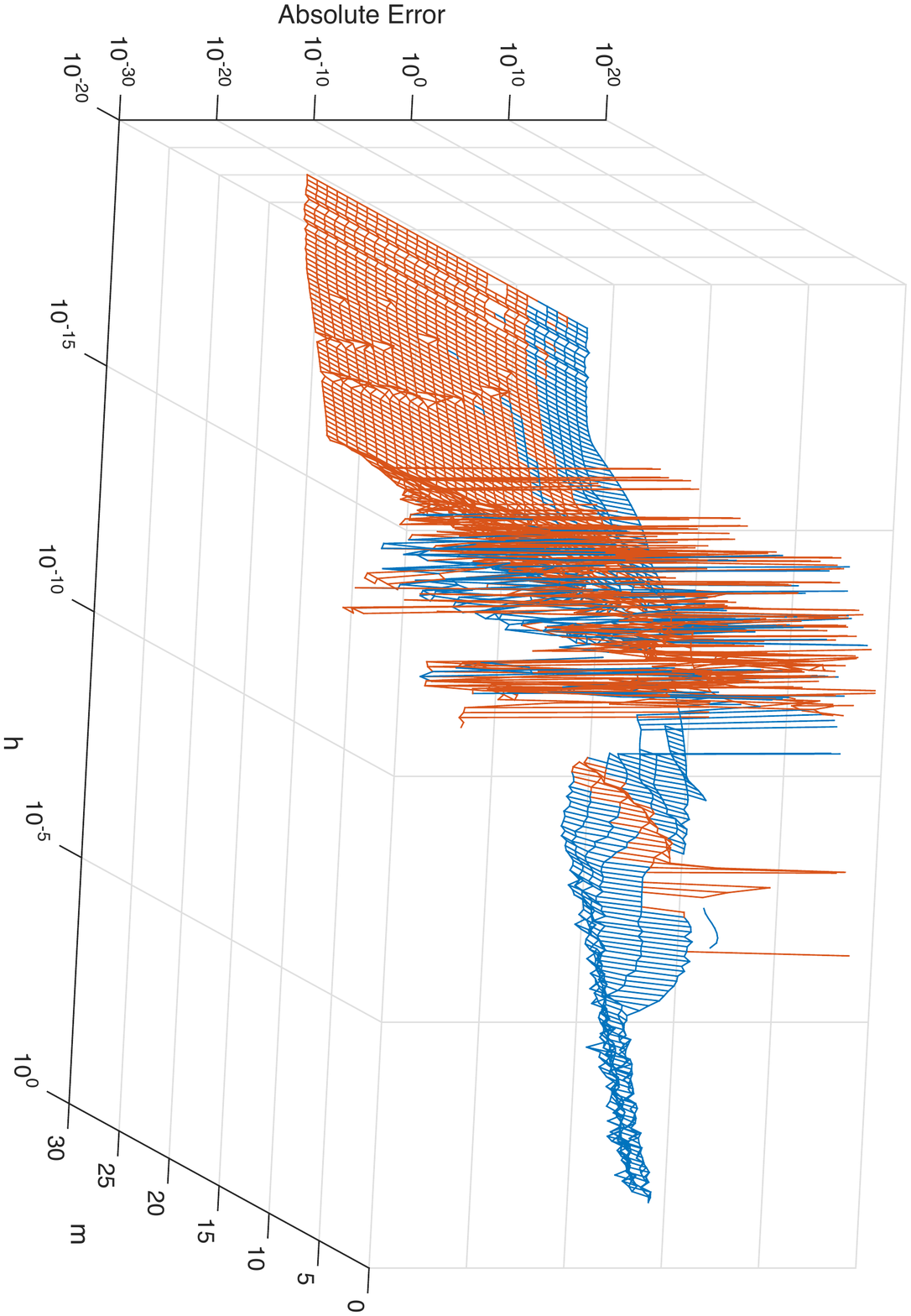}
		\includegraphics[trim = 0.1cm 0.1cm 0.1cm 0.4cm, clip=true, keepaspectratio, width = 0.3\textwidth
	]
	{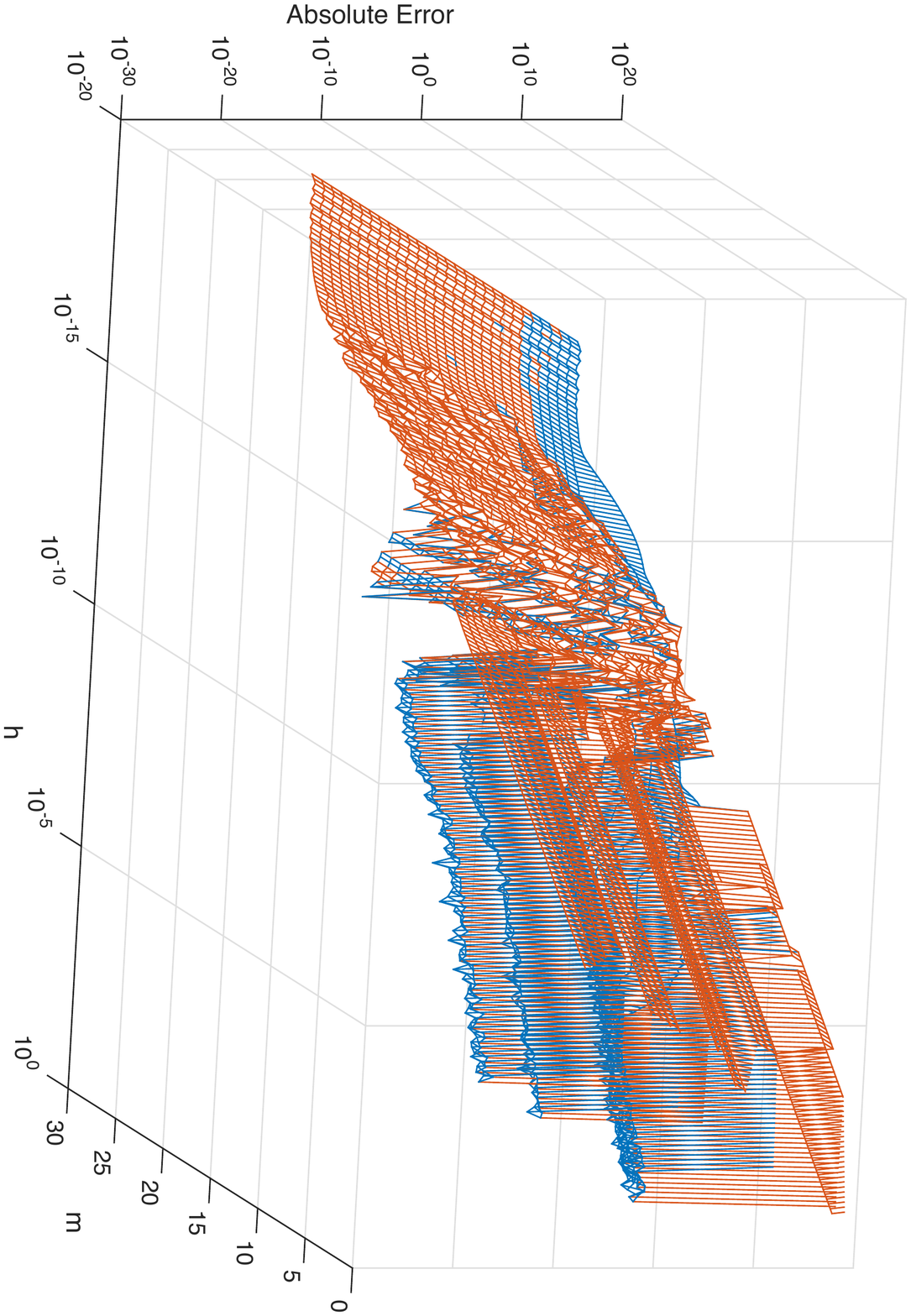}	
	\includegraphics[trim = 0.1cm 0.1cm 0.1cm 0.4cm, clip=true, keepaspectratio, width = 0.3\textwidth
	]
	{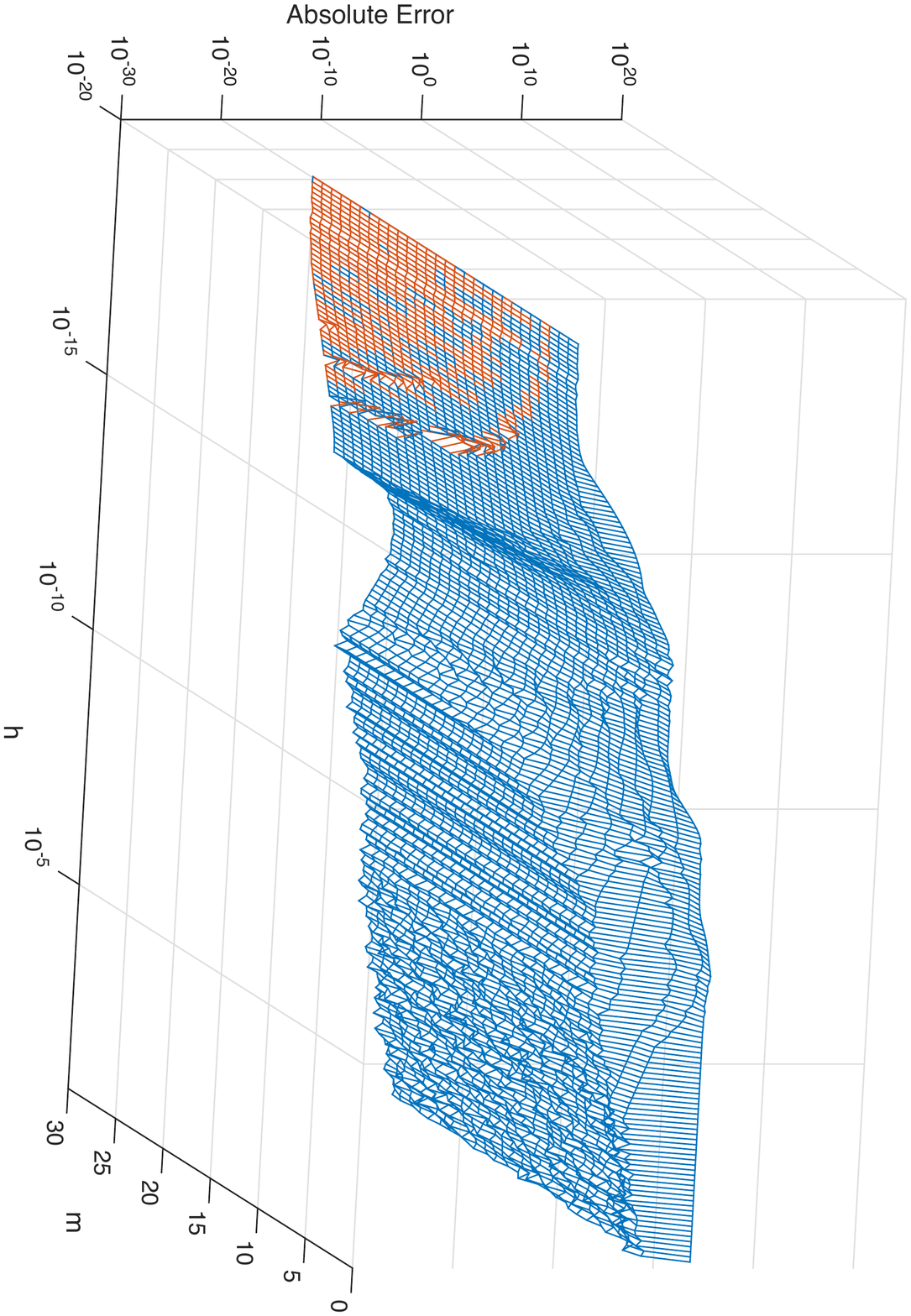}
	\caption{RLC network with $N=507$: Left to right columns show
	the absolute error versus $h$ and $m$ with (a) original Arnoldi process, (b) original Arnoldi process+numerical pruning  and (c) Arnoldi process with explicit structured orthogonalization  +numerical pruning.
	}
	\label{Fig_RC_error_orig}
	\vspace{-0.35cm}
\end{figure}

\appendix
\section{Proofs}
\subsection{Proof of Theorem~\ref{thm4}}

 Introduce  the operator
 \[ \Delta_m:=  (\lambda I-S )^{-1} - W_m (\lambda I-H_m)^{-1} W_m^\top  C .\]  Then the difference between (\ref{eq69}) and (\ref{eq70}) can be bounded by the operator on $v$,
 \begin{eqnarray}
&& \{  f(S_{1,1})  v - W_m  f(H_m) W_m^\top  C V_C v \}\\
&=&\frac{1}{2\pi i}\int_\Gamma f(\lambda) \  \{(\lambda I-S  )^{-1} - W_m (\lambda I-H_m)^{-1} W_m^\top C \}  wd\lambda\\
&=&\frac{1}{2\pi i}\int_\Gamma f(\lambda) \, \Delta_m w d\lambda.\label{eq77}
\end{eqnarray}
By computations, 
\begin{eqnarray}
&&\Delta_m (\lambda I-  S  )  W_m= \{W_m- W_m  (\lambda I-H_m)^{-1}  W_m^\top  C
 (\lambda I- S  )   W_m\}\\
&=& \{ W_m- W_m  (\lambda I-H_m)^{-1}   (\lambda I- H_m ) \}=0,
\end{eqnarray}
and thus
 \beqq \label{eq79}
\Delta_m (w-(\lambda I-S )  W_m y)=\Delta_m w\eeqq
holds  for any vector $y\in \IC^m$.
Note that  columns of $W_m$ lie in the  subspace consisting of vectors \[ \{  S^{k} w: k=0,\ldots, m-1\}. \]   Hence, for each $y\in \IC^m$,
$w-(\lambda I-S)   W_m y$ can be expressed as   $p_m(S; \lambda ) w$ for some polynomial  $p_m(z; \lambda)$ of $z$ with degree $m$. 
Note that  $p_m(\lambda, \lambda)=1$.
 Conversely, for any  (degree $\le m$) polynomial $p_m(z; \lambda)$ with $p_m(\lambda; \lambda)=1$, there exists  some vector $y\in \IC^m$, such that 
\[
w-(\lambda I-S)  W_m y=p_m(S; \lambda).
\] 
All together,  for any $\lambda\in \Gamma$, from (\ref{eq77}) and (\ref{eq79}),
we have
\[
V_C  f(S_{1,1})  v- W_m f(H_m)  W_m^\top C V_C v
=\frac{1}{2\pi i}\int_\Gamma f(\lambda) \Delta_m  p_m(S; \lambda)  w
\,  d\lambda.\]
Choose $\Gamma$ to be the circle with centre $c_0$ and radius $r$ and \[ p_m(z; \lambda)=(\frac{z-c_0}{r})^m.\]
We have
\begin{eqnarray}
&&V_C  f(S_{1,1})  v-V_C V_C^\top  W_m f(H_m)  W_m^\top C V_C v
 =\frac{1}{2\pi i}\int_\Gamma f(\lambda) \Delta_m \left(\frac{S-c_0 I}{r}\right)^m   w d\lambda\\
&=&\frac{(\rho_0/r)^m}{2\pi i}\int_\Gamma f(\lambda) \Delta_m \left(\frac{S-c_0 I}{\rho_0}\right)^m   w d\lambda.
\end{eqnarray}

By Prop.~\ref{Bound1}, $\cF_C(S)$ is bounded by a disk   $\cD(c_0, \rho_0)$.
Then   Prop.~\ref{bound2} and  the power inequality in  theorem in~(\cite{pearcy1966})  indicate 
\[
|\cF_C(\rho_0^{-1}(S-c_0 I))|\le 1,\; |\cF_C((\rho_0^{-1}(S-c_0 I))^m)|\le 1,\]
 and thus \[ \|(\rho_0^{-1}(S-c_0 I))^m u\|_C\le 2 \|u\|_C. 
\]
Hence,  with the aid of  Prop.~\ref{bound3} and (\ref{HS}),
\[
\|(\lambda I-S)^{-1}\|_C\le d(\Gamma, \cF_C(S))^{-1},
\|W_m(\lambda I-H)^{-1} W_m^\top C\|_C\le \|(\lambda I-H)^{-1}\|\le d(\Gamma, \cF_C(S))^{-1},
\]
and then 
\[
\|\Delta_m w\|_C\le 2 d(\Gamma, \cF_C(S))^{-1} \|w\|_C,
\]where $d(\Gamma, \cF_C (S))$ is the shortest distance between $\Gamma$ and $\cF_C(S)$. 
From (\ref{eq79}),
\[
\|\Delta_m  w\|_C=\| \Delta_m p_m(S; \lambda) w\|_C
\le 2 d(\Gamma, \cF_C (S))^{-1}\cdot  \|p_m(S; \lambda) w\|_C \le \frac{2}{r-\rho_0}\cdot 2\left(\frac{\rho_0}{r}\right)^m \| w \|_C,\]
%
%
%
We have for unit vector $v$
 \begin{eqnarray}
&& \|V_C f(S_{1,1})v- W_m f(H_m) W_m^\top C V_C v \|_C \le
 (\max_{\lambda\in \Gamma} |f(\lambda)|)\cdot 2 d(\Gamma, \cF_C(S))^{-1}\cdot    \|p_m(S; \lambda) w\|_C 
\\
&\le&  (\max_{\lambda\in \Gamma} |f(\lambda)|)
 \cdot \frac{4}{r-\rho_0} (\frac{\rho_0}{r})^m.\label{error1}
  \end{eqnarray}

From (\ref{BS}), (\ref{solx3}) and (\ref{x_apr}), the quality of  $x_a$ in (\ref{x_apr}) can be analyzed in the following inequality,
\begin{eqnarray}
\label{x_bd}
&&\| x_\cR (t)-x_a(t)\|_C \le 
  \|\{V_C  f(S_{1,1}) V_C ^\top  -W_m^{(0)} f(H_m^{(0)}) {W_m^{(0)} }^\top  C \} x(0) \|_C \\
  &+&
   \|\{V_C f_1(S_{1,1}) V_C ^\top  -W_m^{(1)}  f_1(H_m^{(1)} ){W_m^{(1)} }^\top  C \}u(0) \|_C \label{eq103}\\
   &+&
  \|\{V_C f_2(S_{1,1}) V_C ^\top  -W_m^{(2)}  f_2(H_m^{(2)} ){W_m^{(2)} }^\top  C \}u'(0) \|_C, \label{eq104}\end{eqnarray}
  which completes the proof.
  
  \subsection{Proof of Prop.~\ref{num}}
  \begin{proof} Derivatives of $\rho, c_0$ with respect to $\gamma$
 are
\beqq
\label{c'}
\frac{dc_0}{d\gamma}=-\frac{1}{2}\left(\frac{\mu_2}{(\gamma+\mu_2)^2}+\frac{\mu_1}{(\gamma+\mu_1)^2}
\right)<0\eeqq
and
  \beqq\label{eq84}
\frac{d\rho}{d\gamma}=\frac{1}{2}\{-\frac{\mu_2}{(\mu_2+\gamma)^2}+\frac{\mu_1}{(\mu_1+\gamma)^2}\}.
\eeqq 
Then 
\begin{eqnarray}
&&\frac{d}{d\gamma} (\log \rho-\log c_0)
=\frac{1}{\rho}\frac{d\rho}{d\gamma}-
\frac{1}{c_0}\frac{dc_0}{d\gamma}
\\
&=& -\left( 
\frac{
\frac{\mu_2}{(\gamma+\mu_2)^2 }-\frac{\mu_1}{(\gamma+\mu_1)^2 }
}{\frac{\mu_2}{(\gamma+\mu_2) }-\frac{\mu_1}{(\gamma+\mu_1) }
}
\right)
+
\left( 
\frac{
\frac{\mu_2}{(\gamma+\mu_2)^2 }+\frac{\mu_1}{(\gamma+\mu_1)^2 }
}{\frac{\mu_2}{(\gamma+\mu_2) }+\frac{\mu_1}{(\gamma+\mu_1) }
}
\right)\\
&=&2\left( 
\frac{
\frac{\mu_2}{(\gamma+\mu_2)^2 }\frac{\mu_1}{(\gamma+\mu_1)^2}(\gamma+\mu_2-\gamma-\mu_1) 
}{(\frac{\mu_2}{(\gamma+\mu_2) })^2-(\frac{\mu_1}{(\gamma+\mu_1) })^2
}
\right)\\
&=&2((\mu_1^{-1}+\mu_2^{-1}) \gamma^2+2\gamma)^{-1}>0.
\end{eqnarray}

\end{proof}

\subsection{Proof of Prop.~\ref{prop3.5}}

\begin{proof}

 By computations, 
\begin{eqnarray}
&& \frac{d}{d\gamma}\log E(\gamma)=\frac{d}{d\gamma} \{\delta (1-\frac{1}{2c_0})+m\log \frac{\rho}{c_0}-\log (c-\rho)\}\\
&=&-\delta (\frac{1}{4c_0^2} (\frac{\mu_2}{(\gamma+\mu_2)^2} + \frac{\mu_1}{(\gamma+\mu_1)^2}  
))+2m ((\frac{1}{\mu_1}+\frac{1}{\mu_2}) \gamma^2+2\gamma)^{-1}+(\mu_1+\gamma)^{-1}\label{eq97}\\
&=&-\delta\frac{\mu_2(\gamma+\mu_1)^2 +\mu_1(\gamma+\mu_2)^2 }{(2\mu_1\mu_2+\gamma(\mu_1+\mu_2))^2}
+\frac{m}{\gamma}  (\frac{2\mu_1\mu_2}{(\mu_1+\mu_2) \gamma+2\mu_1\mu_2})+(\mu_1+\gamma)^{-1}\\
&=&
\xi^{-1}
\{-\delta+\frac{m}{\gamma} 
 (\frac{2\mu_1\mu_2}{(\mu_1+\mu_2) \gamma+2\mu_1\mu_2})\xi
+(\mu_1+\gamma)^{-1}
\xi
 \}.\label{eq102}
\end{eqnarray}
Here  the function $\xi(\gamma)$ introduced   has an upper bound decreasing with respect to $\gamma$,
\begin{eqnarray}
&&\xi(\gamma):=\frac{\mu_1^2(\mu_2+\gamma)^2}{\mu_1(\mu_2+\gamma)^2}\frac{(1+\frac{\mu_2(\mu_1+\gamma)}{\mu_1(\mu_2+\gamma)})^2}{1+\frac{\mu_2(\mu_1+\gamma)^2}{\mu_1(\mu_2+\gamma)^2 }} =
\mu_1\frac{1+2\frac{\mu_2(\mu_1+\gamma)}{\mu_1(\mu_2+\gamma)} +(\frac{\mu_2(\mu_1+\gamma)}{\mu_1(\mu_2+\gamma)})^2}{1+\frac{\mu_2(\mu_1+\gamma)^2}{\mu_1(\mu_2+\gamma)^2 }} 
\label{eq106}\\
&\le& 
\mu_1(1+2(\frac{\mu_2+\gamma}{\mu_1+\gamma}) +\frac{\mu_2}{\mu_1})
=\mu_1+\mu_2+2\frac{\mu_1}{\mu_1+\gamma}(1+\mu_2-\mu_1). \label{eq105}
\end{eqnarray}
Using the AM-GM inequality on the denominator for the second term of (\ref{eq106}), we have one upper bound for $\xi$,
\[\xi(\gamma)\le 
\mu_1(1+\sqrt{2\mu_2/\mu_1}+\mu_2/\mu_1)=(\sqrt{\mu_1}+\sqrt{\mu_2})^2.
\]Hence, 
for  $\gamma\ge \mu_2$, (\ref{eq102}) gives
\[
- \frac{d}{d\gamma}\log E(\gamma)\ge  (\sqrt{\mu_1}+\sqrt{\mu_2})^{-2} \left\{
\delta-\{\frac{2m\mu_1}{\mu_2(3\mu_1+\mu_2)} +\frac{1}{\mu_2+\mu_1} \}(\sqrt{\mu_1}+\sqrt{\mu_2})^2
\right\}.
\]

\end{proof}
\bibliographystyle{alpha}
\newcommand{\etalchar}[1]{$^{#1}$}

%
%
  
\end{document}